\pdfoutput=1
\documentclass[11pt]{article}

\usepackage[T1]{fontenc}
\usepackage[utf8]{inputenc}
\usepackage{lmodern}
\usepackage{amsmath,amssymb,amsthm,mathtools}
\usepackage{microtype}
\usepackage[margin=1in]{geometry}
\usepackage{enumitem}
\usepackage[hidelinks]{hyperref}
\usepackage[nameinlink,noabbrev]{cleveref}
\hypersetup{
  pdftitle={Finite Normal Forms and Causal Inversion for Fractional Operator Algebras},
  pdfauthor={Adel Kassaian},
  pdfsubject={Fractional operator algebras and causal Volterra inversion},
  pdfkeywords={fractional operator algebra, normal form, boundary jets, causal convolution inverse, Sonine kernel, Prabhakar derivative}
}
\numberwithin{equation}{section}
\newtheorem{theorem}{Theorem}
\newtheorem{lemma}{Lemma}
\newtheorem{proposition}{Proposition}
\newtheorem{corollary}{Corollary}
\theoremstyle{definition}
\newtheorem{definition}{Definition}
\newtheorem{example}{Example}
\theoremstyle{remark}
\newtheorem{remark}{Remark}

\newcommand{\C}{\mathbb C}
\newcommand{\R}{\mathbb R}
\newcommand{\Z}{\mathbb Z}
\newcommand{\N}{\mathbb N}
\newcommand{\M}{\mathcal M}
\newcommand{\A}{\mathcal A}
\newcommand{\Pcal}{\mathcal P}
\newcommand{\Vflat}{\mathcal V_{\mathrm{flat}}}
\newcommand{\Kflat}{\mathcal K_{\mathrm{flat}}}
\newcommand{\D}{\mathcal D}
\newcommand{\E}{\mathcal E}
\newcommand{\one}{\mathbf 1}
\newcommand{\dd}{\,\mathrm d}
\newcommand{\supp}{\operatorname{supp}}
\newcommand{\ord}{\operatorname{ord}}
\newcommand{\Span}{\operatorname{span}}
\newcommand{\AC}{\operatorname{AC}}

\title{Finite Normal Forms and Causal Inversion for Fractional\\
Operator Algebras}

\author{Adel Kassaian\\
\small \href{mailto:a.kassaian@gmail.com}{\texttt{a.kassaian@gmail.com}}}
\date{}

\begin{document}
\maketitle

\begin{abstract}
Fractional operator algebras provide a natural setting for linking algebraic normal forms to causal inversion, but arbitrary smooth multipliers and Volterra kernels destroy finite fractional rewriting. We construct an algebra generated by polynomial multipliers, fractional powers indexed by an additive subgroup $\Gamma\subset\R$ containing $\Z$, and flat causal convolution operators. Its finite polynomially weighted diagonal-flat convolution operators form the smallest two-sided ideal containing the flat generators, polynomial multipliers are intrinsic, and distributional separation yields a unique finite normal form with an additive order filtration. For monic normal forms, the leading power selects a canonical causal right factor that extends from the flat core to the natural range $J^\alpha C[0,T]$. Completing the negative-order convolution sector in $L^1$ admits integrable Sonine and Prabhakar kernels, while an independently prescribed continuous bivariate Volterra term remains external. Under injectivity and power-type kernel bounds, variable lower-order terms and this bivariate perturbation reduce to one weakly singular second-kind equation, whose classical resolvent exists on every finite interval without smallness assumptions. A finite boundary-jet block incorporates prescribed initial data without changing the reduced kernel. Sonine, distributed-order, and regularized-Prabhakar realizations follow. In the noninteger Prabhakar regime, an exact range criterion characterizes the classical $\AC^m$ subdomain, and an explicit Riemann--Liouville example proves that it is strictly smaller than the continuous natural range.


\end{abstract}

\section{Introduction}

Operator algebras generated by differentiation, integration and multiplication provide a useful interface between algebraic normal forms and analytic inversion. In the polynomial setting, Bavula \cite{Bavula2011} studied the algebra of integro-differential operators, Regensburger, Rosenkranz and Middeke \cite{Regensburger2009} developed canonical forms through skew-polynomial techniques, and Gao, Keigher and Rosenkranz \cite{Gao2017} investigated related completions. On the concrete space of smooth functions flat at an endpoint, Haghany and Kassaian \cite{Haghany2019} obtained an integer-order normal form, described the units, and applied the resulting algebra to Volterra equations. This raises the question addressed here: can the same finite-normal-form architecture accommodate a group of arbitrary fractional orders while preserving finite rewriting and a stable Volterra remainder?

Two obstructions make the fractional problem different. A general smooth multiplier produces an infinite fractional Leibniz tail, while the full class of smooth Volterra kernels is not, in general, invariant under arbitrary fractional differentiation. We therefore work on smooth causal functions flat at the initial point, restrict multipliers to polynomials, and select a finite polynomially weighted convolution class whose kernels are flat to all orders on the diagonal. On this space the operators $D^\gamma$, $\gamma\in\Gamma$, form a genuine group, polynomial commutation terminates after finitely many steps, and the resulting weighted-convolution Volterra operators form a two-sided invariant ideal. The restrictions are not chosen only to make the proof work: below we show that this ideal is precisely the two-sided ideal generated by the flat convolution operators and that finite reordering with the already present operator $J=D^{-1}$, even modulo an arbitrary smooth diagonal-flat remainder, forces every admitted smooth multiplier to be polynomial. This ideal is denoted $\Vflat$ below; ``diagonal-flat ideal'' is used only as shorthand for this narrower class, not for all smooth bivariate kernels flat on the diagonal. A distributional separation argument then gives uniqueness of the resulting finite normal form and an additive order filtration.

The relation with \cite{Haghany2019} is therefore not a simple inclusion in either direction. The earlier algebra admits broader smooth coefficients and kernels, whereas the present algebra restricts those classes in order to support arbitrary fractional orders, locally finite Leibniz rewriting and a distinguished invariant remainder ideal. Conversely, the present order group is substantially broader and the decomposition into generalized powers modulo the diagonal-flat ideal is unique. The coefficient restriction is structural rather than cosmetic: for a nonpolynomial multiplier such as $p(t)=e^t$, the generalized Leibniz formula contains derivatives of every order and does not terminate. More generally, the finite-closure argument below proves that no nonpolynomial smooth multiplier can be adjoined while retaining finite reordering with $J$ modulo a diagonal-flat remainder.

The fractional literature provides several complementary frameworks. Classical Riemann--Liouville and Caputo operators, semigroup relations and generalized Leibniz formulas are reviewed in \cite{Kilbas2006,Osler1970,Podlubny1999,Samko1993}. Mikusi\'nski's convolutional calculus \cite{Mikusinski1959} and Luchko's operational calculus for Sonine-generated derivatives \cite{Luchko2021,Luchko2022} emphasize convolutional inversion; related developments include the recent extension of Mikusi\'nski operational calculus via Sonine kernels \cite{GuzogluOperational2025}, Sonine-kernel constructions \cite{Ortigueira2024}, first-level derivatives \cite{Alkandari2024}, Prabhakar operational calculi \cite{RaniFernandez2022a,RaniFernandez2022b}, general analytic kernels \cite{Fernandez2019}, and a two-scale Tricomi realization \cite{Colombaro2026}. That operational extension is organized around convolution structures generated by Sonine pairs; the present finite-normal-form algebra instead incorporates polynomial multipliers and a diagonal-flat nonconvolution ideal, while independently prescribed continuous bivariate kernels remain external perturbations coupled through causal factorization. Hilfer and Kleiner \cite{HilferKleiner2024} develop fractional calculus on distributions, while Kleiner and Hilfer \cite{Kleiner2022} obtain normal forms for sequential generalized Riemann--Liouville derivatives through remover and eliminator operators. Their normal forms organize sequential derivative structures, whereas the present one treats polynomial multipliers together with powers indexed by an arbitrary additive subgroup containing $\Z$, modulo a concrete two-sided diagonal-flat Volterra ideal. Durastante et al.\ \cite{Durastante2026} treat nonautonomous fractional equations through a bivariate $\star$-product calculus and a Drazin-type inverse. Thus the present algebra combines variable polynomial coefficients, a group of fractional powers, and an invariant Volterra remainder in one finite coefficient--power normal form.

Recent work has also broadened the kernel side. Al-Shdaifat and Rodr\'iguez-L\'opez \cite{AlShdaifat2025} formulate general fractional integrals and derivatives with genuinely nonconvolution kernels. Their nonconvolution kernel belongs to the fractional operator itself. Here the fractional inverse remains convolutional in the completed smoothing sector, while an independently prescribed continuous bivariate kernel $K(t,s)$ enters as an external Volterra perturbation. This distinction allows the fractional normal-form algebra and the nonconvolution perturbation to be treated separately and then coupled by causal factorization.

These kernel classes are also natural in models with hereditary memory and anomalous transport, where Sonine and Prabhakar operators describe convolutional memory laws \cite{Podlubny1999,Giusti2020}. An independently prescribed $K(t,s)$ allows the perturbing memory to depend on the observation and history times separately, thereby accommodating nonstationary or inhomogeneous effects that cannot be represented by a function of $t-s$ alone.

Flatness is essential for the group law but removes endpoint corrections. The space $\M$ is therefore the homogeneous carrier of the finite-normal-form calculus, not the asserted domain of all equation-level solutions. Initial-value problems require a second, finite layer. For a fixed maximal order $m$, we supplement the homogeneous calculus by the jet functionals $\varepsilon_k u=u^{(k)}(0)$ and the matrix units $e_j\varepsilon_k$, where $e_j(t)=t^j/j!$. They form a finite matrix block that extracts the initial polynomial without changing the normal form on $\M$. Analytically, the same splitting becomes $u=P_{\mathbf d}+Rv$: $P_{\mathbf d}$ carries the prescribed data and $Rv$ the causal dynamics.

The passage from the finite-normal-form layer to the analytic theory is organized around two additional structures: the negative-order causal convolution sector, whose kernel-norm completion is $L^1(0,T)$ and supplies admissible smoothing inverses, and the finite boundary-jet block, which supplies the affine boundary extension. The link is not merely motivational. If a monic normal form has leading order $\alpha$, then the group law on $\M$ gives the canonical right factorization
\begin{equation*}
L_0=\left(I+\sum_j M_{p_j}J^{\alpha-\beta_j}+V_{K_0}J^\alpha\right)D^\alpha.
\end{equation*}
Moreover, $\M=J^\alpha\M\subset J^\alpha C[0,T]$, and the natural-range inverse $(J^\alpha)^{-1}$ restricts to $D^\alpha$ on $\M$. Thus the solution domain may be larger than $\M$ without severing the algebraic origin of the factorization. The completed sector contains every integrable Sonine kernel and, for the Prabhakar family, the fractional-integral expansion converges in kernel norm. If $R=V_r$ is an injective smoothing convolution operator and $A=R^{-1}$ is defined on its natural range, then variable lower-order terms and an independently prescribed continuous kernel $K(t,s)$ reduce to one weakly singular second-kind equation. The two-variable dependence survives explicitly in the reduced kernel, and classical Volterra resolvent theory gives inversion on every finite interval without a norm-smallness assumption. The same reduced kernel governs the boundary-extended problem on $\mathcal P_{m-1}\oplus R(C[0,T])$; the prescribed jets change only the finite-dimensional forcing.

The analytic theory developed below should therefore be viewed as a controlled enlargement initiated by the monic finite-normal-form factorization: the leading power determines the canonical causal right factor, while the subsequent completion and external Volterra perturbations extend the resulting inversion framework beyond the algebra $\A_\Gamma$ itself.

This factorization complements recent equation-level results. Variable-coefficient Prabhakar equations are treated in \cite{FernandezRestrepo2022}, Volterra--Prabhakar operational structures in \cite{Tomovski2026}, variable-coefficient Sonine equations in \cite{Guzoglu2026}, and the two-scale Sonine calculus in \cite{Colombaro2026}. Eshaghi and Ordokhani \cite{Eshaghi2021} consider a regularized-Prabhakar equation with a two-variable kernel and specialize that kernel to convolution form for Laplace--Bromwich inversion. In contrast, the reduction below retains an independently prescribed continuous bivariate kernel $K(t,s)$ explicitly, together with variable lower-order coefficients and compatible lower general derivatives.

The resulting equation-level gain can be stated concretely. In the Sonine setting, the framework combines continuous variable coefficients and compatible lower Sonine derivatives with an independently prescribed continuous bivariate Volterra memory term. Its distributed-order realization admits genuinely continuous positive order measures, rather than only finite sums of orders, under the admissibility hypotheses verified below. In the noninteger regularized-Prabhakar setting, the same reduction also incorporates arbitrary prescribed initial jets and distinguishes unconditional solvability in the affine natural range from the additional compatibility required for a classical $\AC^m$ solution. In each case the output is a resolvent representation of the unique natural-range solution, not in general an elementary closed-form formula.

The two principal theorem-level contributions are the unique finite coefficient--power normal form with its two-sided finite polynomially weighted diagonal-flat convolution ideal and the exact classical-range criterion in the regularized Prabhakar case. The core-to-natural-range proposition integrates these two parts: the leading power selected by the finite normal form produces the causal right factor, while the natural-range realization extends that identity beyond $\M$. This bridge is proved for monic normal forms, whereas the algebraic normal-form theorem itself has no monicity hypothesis. The finite boundary-jet block, the completed smoothing sector, and the causal factorization then connect the algebraic result to inversion while preserving an independently prescribed continuous bivariate perturbation. The proof combines kernel-norm density, weakly singular Volterra resolvents, the Prabhakar semigroup identity, and Bernstein--Stieltjes theory for distributed order. The final range criterion supplies the exact compatibility condition on the causal density that upgrades the unconditional natural-range solution to a classical $\AC^m$ solution.

The main results can be summarized in three steps.
\begin{enumerate}[leftmargin=*,label=\arabic*.]
\item For an additive subgroup $\Gamma\subset\R$ containing $\Z$, every element of the algebra generated by polynomial multipliers, $D^\gamma$ ($\gamma\in\Gamma$), and flat causal convolution operators has a unique finite normal form
\begin{equation*}
A=\sum_{\gamma\in F}p_\gamma(t)D^\gamma+V_K,
\qquad F\subset\Gamma\ \text{finite}.
\end{equation*}
The proof uses kernel detection and separation of generalized powers. The coefficient and remainder classes also satisfy sharp closure properties: polynomials are forced by finite reordering with $J$ modulo a diagonal-flat remainder, and $\Vflat$ is the smallest two-sided ideal containing the flat convolution generators. The induced order is additive, so the quotient by the diagonal-flat ideal has no zero divisors. A finite boundary matrix block then records prescribed endpoint jets without enlarging the homogeneous normal-form algebra.

\item A monic finite normal form with leading term $D^\alpha$ factors
canonically on $\M$ as $L_0=(I+W_0)D^\alpha$.  The identity
$\M=J^\alpha\M\subset J^\alpha C[0,T]$ shows that
$(I+W_0)(J^\alpha)^{-1}$ is a compatible natural-range extension, so the
equation-level theory retains a precise algebraic origin even though its
solution domain is larger than $\M$.  The kernel completion of the
negative-order convolution sector is then the Banach algebra $L^1(0,T)$ under convolution. In particular,
\begin{equation*}
\E_{\rho,\mu,\omega}^{\gamma}
=\sum_{n=0}^{\infty}\frac{(\gamma)_n\omega^n}{n!}J^{\mu+\rho n}
\end{equation*}
converges in kernel norm. More generally, if $B_jR=V_{q_j}$ for an injective smoothing inverse $R$, then
\begin{equation*}
L=A+\sum_{j=1}^{N}p_j(t)B_j+V_K=(I+W_L)A,
\qquad A=R^{-1},
\end{equation*}
and the explicit kernel of $W_L$ is generally nonconvolution. Its Volterra resolvent yields a resolvent representation of the inverse on every finite interval. The boundary extension acts on $\mathcal P_{m-1}\oplus R(C[0,T])$ and uses the same reduced kernel.

\item The abstract factorization yields finite-order, Sonine, distributed-order and regularized Prabhakar realizations. For a Sonine pair $(\kappa,k)$ it gives variable-coefficient equations with an independently prescribed continuous bivariate perturbation. In the noninteger regularized Prabhakar regime $m-1<\mu<m$, arbitrary finite initial jets are incorporated directly into the affine natural domain. The new range statement is not the standard Prabhakar left-inverse identity: it is the necessary-and-sufficient compatibility condition characterizing precisely which natural-range solutions belong to the classical $\AC^m[0,T]$ domain. An explicit sufficient condition for the causal density and a strict-inclusion example in the Riemann--Liouville subcase show that the criterion has genuine domain content.
\end{enumerate}

Section~2 develops the finite normal form, the flat homogeneous core and the boundary block. Section~3 first derives the core-to-natural-range factorization of a monic normal form, then gives the convolution completion and the homogeneous and boundary-extended causal factorizations. Section~4 specializes the framework to power-law, Sonine, distributed-order and Prabhakar operators. Section~5 summarizes the structural consequences and limitations.

\section{Finite normal forms in the fractional causal calculus}

\subsection{Flat causal functions and convolution injectivity}

Fix $T>0$ and set
\begin{equation*}
\M=\M[0,T]:=\{f\in C^\infty[0,T]:f^{(n)}(0)=0\text{ for every }n\ge0\}.
\end{equation*}
It is a Fr\'echet space for the seminorms
\begin{equation*}
p_m(f)=\sum_{j=0}^{m}\sup_{0\le t\le T}|f^{(j)}(t)|,
\qquad m\in\Z_+.
\end{equation*}
For $\mu>0$ let
\begin{equation*}
(J^\mu f)(t)=\frac1{\Gamma(\mu)}\int_0^t(t-s)^{\mu-1}f(s)\dd s.
\end{equation*}
On $\M$, the Riemann--Liouville and Caputo derivatives coincide, all initial correction terms vanish, and the operators form a group under composition. We denote their common value by $D^\alpha$, set $D^{-\mu}=J^\mu$, and use
\begin{equation}
D^\alpha D^\beta=D^{\alpha+\beta},
\qquad \alpha,\beta\in\R.
\label{eq:group-law}
\end{equation}
These facts are standard on flat causal functions; see \cite{Kilbas2006,Podlubny1999,Samko1993}.

For kernel calculations, let
\begin{equation*}
h_\beta(t)=\frac{t_+^{\beta-1}}{\Gamma(\beta)},
\end{equation*}
continued analytically in $\beta$ as a causal distribution. Then
\begin{equation*}
h_\alpha*h_\beta=h_{\alpha+\beta},
\qquad D^\alpha f=h_{-\alpha}*f,
\end{equation*}
for flat causal $f$; see \cite{Gelfand1964}. The generalized Leibniz formula is
\begin{equation}
D^\alpha(pf)=\sum_{j=0}^{\infty}\binom{\alpha}{j}p^{(j)}D^{\alpha-j}f,
\qquad
\binom{\alpha}{0}=1,
\quad
\binom{\alpha}{j}=\frac{\alpha(\alpha-1)\cdots(\alpha-j+1)}{j!}.
\label{eq:leibniz}
\end{equation}
This polynomial definition is valid for every $\alpha\in\R$, including negative integers; whenever the gamma quotients are finite, it agrees with $\Gamma(\alpha+1)/[\Gamma(j+1)\Gamma(\alpha-j+1)]$. For $p\in\C[t]$, the sum terminates. This is the local finiteness mechanism used below.

The stability of the flat space under every real fractional order will be used in the ideal argument. We record it explicitly.

\begin{lemma}[Invariance of the flat causal space]\label{lem:flat-invariance}
For every $\alpha\in\R$, the operator $D^\alpha$ maps $\M$ bijectively onto $\M$, with inverse $D^{-\alpha}$.
\end{lemma}

\begin{proof}
First let $\beta>0$ and $f\in\M$. Because all endpoint terms vanish, differentiation under causal convolution gives
\begin{equation*}
(J^\beta f)^{(m)}=J^\beta f^{(m)},
\qquad m\in\Z_+.
\end{equation*}
For every $m,N\in\Z_+$, flatness gives $|f^{(m)}(s)|\le C_{m,N}s^N$ near the origin. Hence
\begin{equation*}
|(J^\beta f)^{(m)}(t)|
\le \frac{C_{m,N}}{\Gamma(\beta)}\int_0^t(t-s)^{\beta-1}s^N\dd s
=C'_{m,N,\beta}t^{N+\beta}.
\end{equation*}
Thus $J^\beta f$ is smooth and flat at the origin, so $J^\beta\M\subset\M$. Given $\alpha\in\R$, choose an integer $n\ge0$ with $n>\alpha$. On flat functions,
\begin{equation*}
D^\alpha f=J^{n-\alpha}f^{(n)}\in\M.
\end{equation*}
Applying the same conclusion to $-\alpha$ and using \eqref{eq:group-law} shows that $D^{-\alpha}$ is the inverse of $D^\alpha$ on $\M$.
\end{proof}

\begin{remark}[Role of the flat core]\label{rem:flat-core}
The space $\M$ is not introduced as the space of all solutions of the equations
considered below.  It is the common homogeneous core on which every
$D^\gamma$, $\gamma\in\Gamma$, is a bijection and the group law
\eqref{eq:group-law} is free of endpoint corrections.  As in the integer-order
construction of \cite{Haghany2019}, it is maximal for the elementary
two-sided $D$--$J$ calculus: if a linear subspace of $C^\infty[0,T]$ is
invariant under $D$ and $J=J^1$ and $JD=DJ=I$ there, then $JDf=f$ gives
$f(0)=0$, and applying the same identity to $D^nf$ gives
$f^{(n)}(0)=0$ for every $n$.  Thus that subspace is contained in $\M$.
The natural ranges $R(C[0,T])$ used in Sections~3--4 are larger analytic
domains.  Their relation to this core is an extension question, made explicit
for the power-law sector in \Cref{prop:core-range-bridge}; prescribed endpoint
data are handled separately by the boundary block of Section~2.4.
\end{remark}

We shall repeatedly use injectivity of causal convolution on a finite interval. The following form makes the natural-range arguments independent of the classical domain of any associated derivative.

\begin{lemma}[Causal convolution injectivity]\label{lem:injectivity}
Let $r\in L^1(0,T)$ and suppose that $r$ does not vanish almost everywhere on any interval $(0,\varepsilon)$, $\varepsilon>0$. Then
\begin{equation}
(V_rf)(t)=\int_0^t r(t-s)f(s)\dd s
\label{eq:Vr}
\end{equation}
defines an injective operator on both $L^1(0,T)$ and $C[0,T]$.
\end{lemma}

\begin{proof}
Extend $r$ and $f$ by zero to the negative half-line. If $V_rf=0$ on $(0,T)$, the local Titchmarsh convolution theorem \cite{Titchmarsh1948} gives numbers $a,b\ge0$ with $a+b\ge T$ such that $r=0$ almost everywhere on $(0,a)$ and $f=0$ almost everywhere on $(0,b)$. The hypothesis forces $a=0$, hence $b\ge T$. Thus $f=0$ almost everywhere on $(0,T)$; for continuous $f$, this also gives pointwise vanishing on $[0,T]$.
\end{proof}

\subsection{The fractional operator algebra and the weighted-convolution ideal}

Let $\Pcal=\C[t]$ and let $\Gamma\subset\R$ be an additive subgroup containing $\Z$. For $p\in\Pcal$, write $M_p$ for multiplication by $p$. For $k\in\M$, define $V_k$ by \eqref{eq:Vr}, and introduce the finite kernel class
\begin{equation}
\Kflat=\left\{K(t,s)=\sum_{r=1}^{N}p_r(t)k_r(t-s):p_r\in\Pcal,\ k_r\in\M,\ N<\infty\right\}.
\label{eq:Kflat}
\end{equation}
Thus $\Kflat$ consists specifically of finite polynomially weighted convolution kernels. Every $K\in\Kflat$ is smooth on $\Delta_T=\{(t,s):0\le s\le t\le T\}$ and flat to all orders in the transverse variable $t-s$ along the diagonal; no claim is made here that every smooth bivariate kernel with diagonal flatness belongs to $\Kflat$. Accordingly, $\Vflat$ below always means the finite polynomially weighted diagonal-flat convolution ideal determined by \eqref{eq:Kflat}.

\begin{definition}
The fractional integro-differential algebra $\A_\Gamma$ is the subalgebra of $\operatorname{End}_{\C}(\M)$ generated by
\begin{equation*}
\{M_p:p\in\Pcal\},
\qquad \{D^\gamma:\gamma\in\Gamma\},
\qquad \{V_k:k\in\M\}.
\end{equation*}
\end{definition}

The following identities are rewrite rules on $\M$:
\begin{align}
D^\gamma M_p&=\sum_{j=0}^{\deg p}\binom{\gamma}{j}M_{p^{(j)}}D^{\gamma-j},
\label{eq:rewrite1}\\
V_kM_p&=\sum_{j=0}^{\deg p}\frac{(-1)^j}{j!}M_{p^{(j)}}V_{u^jk(u)},
\label{eq:rewrite2}\\
D^\gamma V_k&=V_kD^\gamma=V_{D^\gamma k},
\label{eq:rewrite3}\\
V_kV_\ell&=V_{k*\ell}.
\label{eq:rewrite4}
\end{align}
Here \eqref{eq:rewrite1} follows from \eqref{eq:leibniz}, \eqref{eq:rewrite2} from the exact Taylor formula $p(s)=\sum_jp^{(j)}(t)(s-t)^j/j!$, and \eqref{eq:rewrite3}--\eqref{eq:rewrite4} from causal convolution.

\begin{proposition}\label{prop:ideal}
The set
\begin{equation*}
\Vflat:=\{V_K:K\in\Kflat\}
\end{equation*}
is a two-sided ideal of $\A_\Gamma$.
\end{proposition}

\begin{proof}
It is enough to check stability under the generators. Left multiplication by $M_p$ merely replaces each coefficient $p_r(t)$ in \eqref{eq:Kflat} by $p(t)p_r(t)$. Right multiplication by $M_p$ is reduced by \eqref{eq:rewrite2}; every new kernel $u^jk_r(u)$ remains in $\M$. Left or right multiplication by $D^\gamma$ is reduced by \eqref{eq:rewrite1} and \eqref{eq:rewrite3}, and the preceding lemma gives $D^\gamma k_r\in\M$.

For completeness, consider two elementary weighted convolution kernels. Their composition has kernel
\begin{equation*}
p(t)\int_s^t k(t-\xi)q(\xi)\ell(\xi-s)\dd\xi.
\end{equation*}
Expanding $q(\xi)$ about $t$ by its finite Taylor formula turns this into
\begin{equation*}
\sum_{j=0}^{\deg q}\frac{(-1)^j}{j!}p(t)q^{(j)}(t)
\big[(u^jk(u))*\ell\big](t-s).
\end{equation*}
Each convolution kernel is flat because $\M$ is closed under multiplication by monomials and causal convolution. Finite linearity proves the assertion.
\end{proof}

\begin{corollary}[Minimality of the Volterra remainder]\label{cor:minimal-ideal}
The ideal $\Vflat$ is exactly the smallest two-sided ideal of $\A_\Gamma$
that contains all flat convolution generators $V_k$, $k\in\M$.
\end{corollary}

\begin{proof}
Let $\mathcal I$ be any two-sided ideal containing every $V_k$, $k\in\M$.
Then $M_pV_k\in\mathcal I$ for every $p\in\Pcal$, and hence every finite
sum $V_K$ with $K$ of the form \eqref{eq:Kflat} belongs to $\mathcal I$.
Thus $\Vflat\subseteq\mathcal I$. The reverse minimality statement follows
from \Cref{prop:ideal}, which shows that $\Vflat$ itself is a two-sided ideal
containing all the generators $V_k$.
\end{proof}

\begin{remark}[Finite closure forces polynomial multipliers]
The multiplier restriction is forced by the requested finite normal-form
architecture. More generally, if $JM_a$ admitted any finite coefficient--power
normal form modulo a smooth diagonal-flat remainder, restriction to $s<t$ and
separation of the distinct generalized powers would eliminate every exponent
except the negative integers $-1,-2,\ldots$; the nonnegative integer powers,
which are supported on the diagonal, would vanish as well. Thus such a form
necessarily reduces to the following situation. Let $a\in C^\infty[0,T]$ and
suppose that, for some $N<\infty$ and
$b_0,\ldots,b_N\in C^\infty[0,T]$, one has
\begin{equation}
JM_a=\sum_{j=0}^{N}M_{b_j}J^{j+1}+V_K,
\label{eq:finite-J-closure}
\end{equation}
where $K$ is smooth and flat to every order in $t-s$ along the diagonal.
Equality of the kernels for $s<t$, with $u=t-s$, gives
\begin{equation*}
a(t-u)=\sum_{j=0}^{N}b_j(t)\frac{u^j}{j!}+K(t,t-u).
\end{equation*}
Taking successive $u$-derivatives at $u=0$ yields
$b_j(t)=(-1)^ja^{(j)}(t)$ for $0\le j\le N$ and
$a^{(j)}(t)=0$ for every $j>N$. Hence $a$ is a polynomial of degree at most
$N$. Conversely, the finite Taylor formula gives \eqref{eq:finite-J-closure}
with $K=0$ for every polynomial $a$.  Thus, even if the remainder in
\eqref{eq:finite-J-closure} is allowed to range over the full class of smooth
diagonal-flat kernels rather than only $\Kflat$, the polynomial multipliers
are exactly the smooth multipliers that reorder finitely with the already
present operator $J=D^{-1}$. Together with \Cref{cor:minimal-ideal}, this shows
that the two restricted classes are closure-determined rather than ad hoc.
\end{remark}

\subsection{Kernel separation and unique normal form}

The next observation detects the entire causal distribution kernel, including components supported on the diagonal.

\begin{lemma}[Kernel detection]\label{lem:kernel-detection}
Let $I\Subset(0,T)$ be an open interval and let $Q\in\D'(I\times I)$ have support in $\{(t,s):s\le t\}$. Define its kernel action by
\begin{equation*}
\langle T_Q\psi,\varphi\rangle:=\langle Q,\varphi(t)\psi(s)\rangle,
\qquad \varphi,\psi\in C_c^\infty(I).
\end{equation*}
If $T_Q\psi=0$ for every $\psi\in C_c^\infty(I)$, then $Q=0$ in $\D'(I\times I)$. Consequently, two causal distribution kernels that induce the same operator on $\M$ agree locally across every interior portion of the diagonal, not merely on the open set $s<t$.
\end{lemma}

\begin{proof}
Every $\psi\in C_c^\infty(I)$, extended by zero near the origin, belongs to $\M$. The hypothesis therefore gives
\begin{equation*}
\langle Q,\varphi\otimes\psi\rangle=0
\qquad\text{for all }\varphi,\psi\in C_c^\infty(I).
\end{equation*}
Finite sums of tensor products are dense in $C_c^\infty(I\times I)$ in its standard test-function topology. Continuity of $Q$ then yields $\langle Q,\Phi\rangle=0$ for every $\Phi\in C_c^\infty(I\times I)$, hence $Q=0$. Because these test functions range over the full square, the argument also detects distributions such as $\delta(t-s)$ and its derivatives supported on the diagonal.
\end{proof}

\begin{lemma}[Separation of generalized powers]\label{lem:separation}
Let $F\subset\R$ be finite. Suppose
\begin{equation}
\sum_{\gamma\in F}a_\gamma(t)h_{-\gamma}(t-s)+R(t,s)=0
\label{eq:separation}
\end{equation}
as a distribution kernel near an interior point of the diagonal, where the $a_\gamma$ are smooth and $R$ is smooth and flat to every order in $t-s$. Then all $a_\gamma$ and $R$ vanish.
\end{lemma}

\begin{proof}
Recall explicitly that $h_{-n}=\delta^{(n)}$ for $n\in\Z_{\ge0}$, whereas for $\gamma\notin\Z_{\ge0}$ the restriction of $h_{-\gamma}$ to $u>0$ is the function $u^{-\gamma-1}/\Gamma(-\gamma)$.
Use the local coordinates $(t,u)=(t,t-s)$. Restriction of a distribution to the open set $u>0$ is legitimate and discards precisely the components supported on $u=0$. Thus, after restricting \eqref{eq:separation} to $u>0$, all terms with $\gamma\in\Z_{\ge0}$ disappear because $h_{-\gamma}=\delta^{(\gamma)}$ is supported on $u=0$. Every remaining term is a smooth function on $u>0$ and has the form $a_\gamma(t)c_\gamma u^{-\gamma-1}$ with $c_\gamma=1/\Gamma(-\gamma)\ne0$.

Fix an interior $t_0$ and suppose some remaining $a_\gamma(t_0)$ is nonzero. Choose the largest such index $\gamma_*$. Evaluating at $(t_0,u)$, dividing by $u^{-\gamma_*-1}$ and letting $u\downarrow0$ leaves $c_{\gamma_*}a_{\gamma_*}(t_0)\ne0$: every lower generalized power contributes a factor $u^{\gamma_*-\gamma}\to0$. For the remainder, choose an integer $N>-\gamma_*-1$; flatness gives $R(t_0,t_0-u)=O(u^N)$, so the divided remainder is $O(u^{N+\gamma_*+1})\to0$. This contradiction proves that all noninteger and negative-integer coefficients vanish. Thus negative integer orders are separated on $u>0$ together with the other nondiscrete powers, whereas nonnegative integer orders remain in the diagonal-supported block below. The same restricted identity then gives $R=0$ for $u>0$, and smoothness together with transverse flatness extends this equality across $u=0$.

It remains only to separate the diagonal-supported terms
\begin{equation*}
\sum_{n\in F\cap\Z_{\ge0}}a_n(t)\delta^{(n)}(u)=0.
\end{equation*}
Fix $N\in F\cap\Z_{\ge0}$. Choose $\psi_N\in C_c^\infty(\R)$, supported in the coordinate neighborhood, such that
\begin{equation*}
\psi_N^{(j)}(0)=\delta_{jN}
\qquad\text{for every }j\in F\cap\Z_{\ge0};
\end{equation*}
for example, one may take $\psi_N(u)=u^N\chi(u)/N!$, where $\chi$ is identically one near $u=0$. Testing against $\varphi(t)\psi_N(u)$, with arbitrary $\varphi\in C_c^\infty(I)$, gives
\begin{equation*}
0=(-1)^N\int_I a_N(t)\varphi(t)\dd t.
\end{equation*}
Hence $a_N=0$ as a distribution, and therefore as a smooth function, on $I$. Since $N$ was arbitrary, every diagonal coefficient vanishes. Equivalently, the distributions $\delta^{(n)}(u)$ are linearly independent over smooth functions of the tangential variable. Thus no diagonal term can cancel another diagonal order or the smooth-flat remainder, completing the proof.
\end{proof}

\begin{theorem}[Finite normal form]\label{thm:normal-form}
Every $A\in\A_\Gamma$ has a unique representation
\begin{equation}
A=\sum_{\gamma\in F}p_\gamma(t)D^\gamma+V_K,
\label{eq:normal-form}
\end{equation}
where $F\subset\Gamma$ is finite, $p_\gamma\in\Pcal$, and $K\in\Kflat$.
\end{theorem}

\begin{proof}
For existence, write $A$ as a finite linear combination of words in the generators. Any word containing a Volterra factor belongs to $\Vflat$ by \Cref{prop:ideal}, so it remains to normalize a word formed only from multipliers and fractional powers. Combine adjacent multipliers and adjacent powers whenever they occur. For such a word, let its inversion number be the number of pairs in which a power $D^\eta$ occurs to the left of a multiplier $M_p$. If the inversion number is positive, the word contains an adjacent subword $D^\eta M_p$. Applying \eqref{eq:rewrite1} replaces it by the finite sum
\begin{equation*}
\sum_{j=0}^{\deg p}\binom{\eta}{j}M_{p^{(j)}}D^{\eta-j}.
\end{equation*}
In every nonzero summand the indicated power has moved past that multiplier, so the inversion number decreases by one; combining adjacent powers by \eqref{eq:group-law} cannot increase it. Induction on the inversion number therefore terminates after finitely many steps and leaves a finite sum of terms $M_pD^\gamma$. Summing equal orders gives \eqref{eq:normal-form}, with all words containing a Volterra factor collected into one element of $\Vflat$.

For uniqueness, subtract two representations and denote the resulting causal distribution kernel by $Q$. For every $I\Subset(0,T)$, equality of the operators on $\M$ implies that $T_Q$ annihilates $C_c^\infty(I)$. \Cref{lem:kernel-detection} therefore gives $Q=0$ on the full square $I\times I$, including its diagonal. Near each interior diagonal point this identity has the form \eqref{eq:separation}, so \Cref{lem:separation} forces every coefficient and the flat remainder to vanish on $(0,T)$. Polynomiality of the coefficients and continuity of the flat kernel extend the conclusion to the endpoints.
\end{proof}

\begin{remark}[Comparison with earlier normal forms]\label{rem:normal-form-comparison}
The novelty of \Cref{thm:normal-form} is not the word ``normal form'' or any one rewrite identity in isolation. The integer-order algebra of \cite{Haghany2019} permits broader smooth coefficients and kernels but does not contain a group of arbitrary fractional powers. The sequential normal forms of \cite{Kleiner2022} organize generalized Riemann--Liouville derivative strings by remover and eliminator operators, but do not give the present polynomial coefficient--power decomposition modulo $\Vflat$. The theorem here is the combination of polynomial multipliers, powers indexed by an arbitrary additive subgroup $\Gamma\supset\Z$, the concrete two-sided ideal $\Vflat$, and uniqueness of the resulting finite decomposition. The closure results preceding the theorem further show that the multiplier class is maximal among smooth coefficients for finite reordering with $J$ modulo a diagonal-flat remainder and that $\Vflat$ is the minimal two-sided remainder ideal generated by the flat convolutions. No inclusion between these earlier algebras and $\A_\Gamma$ is asserted.
\end{remark}

\begin{example}[A normal-ordering calculation]\label{ex:normal-ordering}
Let $\alpha,\beta,\gamma\in\Gamma$ and $k\in\M$. The rewrite rules give
\begin{align*}
D^\alpha M_{t^2}D^\beta+M_tV_kD^\gamma
={}&M_{t^2}D^{\alpha+\beta}
+2\alpha M_tD^{\alpha+\beta-1}
+\alpha(\alpha-1)D^{\alpha+\beta-2}
+M_tV_{D^\gamma k}.
\end{align*}
The Leibniz tail stops after the second derivative of $t^2$, while $D^\gamma k\in\M$ and hence $M_tV_{D^\gamma k}\in\Vflat$. Thus the first three terms are the polynomial--fractional part of the normal form and the last term is its diagonal-flat Volterra remainder.
\end{example}

For $A\notin\Vflat$ in the form \eqref{eq:normal-form}, define
\begin{equation*}
\ord_\Gamma A=\max\{\gamma:p_\gamma\ne0\},
\end{equation*}
and put $\ord_\Gamma A=-\infty$ on $\Vflat$.

\begin{corollary}[Additivity and induced order filtration]\label{cor:order-filtration}
If $A,B\notin\Vflat$, then
\begin{equation}
\ord_\Gamma(AB)=\ord_\Gamma A+\ord_\Gamma B.
\end{equation}
For $\lambda\in\R$, set
\begin{equation*}
\mathcal F_\lambda=\{A\in\A_\Gamma:\ord_\Gamma A\le\lambda\}.
\end{equation*}
Then $(\mathcal F_\lambda)_{\lambda\in\R}$ is an increasing filtration by complex vector subspaces and
\begin{equation*}
\mathcal F_\lambda\mathcal F_\mu\subseteq\mathcal F_{\lambda+\mu}.
\end{equation*}
\end{corollary}

\begin{proof}
If the leading terms are $p(t)D^\alpha$ and $q(t)D^\beta$, the unique term of order $\alpha+\beta$ is $p(t)q(t)D^{\alpha+\beta}$. Every Leibniz correction has lower order, and $pq\ne0$ in $\C[t]$.

The normal form also gives $\ord_\Gamma(A+B)\le\max\{\ord_\Gamma A,\ord_\Gamma B\}$, with strict inequality allowed when leading terms cancel, and scalar multiplication does not increase order. Hence every $\mathcal F_\lambda$ is a vector subspace and the family is increasing. The product inclusion follows from the additivity relation in \Cref{cor:order-filtration} when neither factor lies in $\Vflat$; if one factor lies in $\Vflat$, the ideal property gives $AB\in\Vflat$ and the conclusion remains true.
\end{proof}

\begin{corollary}[Quotient domain and multiplicative leading symbol]\label{cor:quotient-domain}
The quotient algebra $\A_\Gamma/\Vflat$ has no zero divisors. Equivalently, if $A,B\notin\Vflat$, then $AB\notin\Vflat$. Moreover, if
\begin{align*}
A&=p(t)D^\alpha+\text{terms of order }<\alpha \pmod{\Vflat},\\
B&=q(t)D^\beta+\text{terms of order }<\beta \pmod{\Vflat},
\end{align*}
then the leading term of $AB$ modulo $\Vflat$ is $p(t)q(t)D^{\alpha+\beta}$.
\end{corollary}

\begin{proof}
If the classes of $A$ and $B$ are nonzero, then $A,B\notin\Vflat$. By \Cref{cor:order-filtration}, $\ord_\Gamma(AB)=\ord_\Gamma A+\ord_\Gamma B$ is finite, so $AB\notin\Vflat$. The leading-term identity follows from the leading-order calculation in the proof of \Cref{cor:order-filtration}.
\end{proof}

\begin{remark}
The restrictions in the definition are structural. A nonpolynomial smooth multiplier generally produces an infinite fractional Leibniz tail; for example, $p(t)=e^t$ has $p^{(j)}=e^t$ for every $j$, so the generalized Leibniz expansion does not terminate for a noninteger order. Likewise, the full class of smooth Volterra kernels is not, in general, invariant under fractional differentiation. The algebra $\A_\Gamma$ therefore isolates a locally finite coefficient core together with a two-sided Volterra ideal stable under every $D^\gamma$, $\gamma\in\Gamma$.
\end{remark}

\subsection{Finite boundary jets and the homogeneous block}

The flat space $\M$ removes all endpoint corrections and is therefore the natural carrier of the group algebra above. Initial-value problems of maximal order $m$, however, involve only the finite jet of orders $0,\ldots,m-1$. This finite quotient has a canonical algebraic realization.

Fix $m\in\N$ and set
\begin{equation*}
e_k(t)=\frac{t^k}{k!},
\qquad
\varepsilon_k u=u^{(k)}(0),
\qquad 0\le k<m,
\end{equation*}
on $C^{m-1}[0,T]$. Define the boundary matrix units and the jet projection by
\begin{equation}
E_{jk}=e_j\varepsilon_k,
\qquad
\Pi_m=\sum_{k=0}^{m-1}E_{kk},
\qquad
Q_m=I-\Pi_m,
\label{eq:jet-projectors}
\end{equation}
and put
\begin{equation*}
\mathcal P_{m-1}=\Span\{e_0,\ldots,e_{m-1}\},
\qquad
\mathcal N_m=\bigcap_{k=0}^{m-1}\ker\varepsilon_k.
\end{equation*}

\begin{proposition}[Finite boundary-jet decomposition]\label{prop:boundary-jets}
The operators $E_{jk}$ satisfy
\begin{equation}
E_{ij}E_{k\ell}=\delta_{jk}E_{i\ell},
\qquad 0\le i,j,k,\ell<m.
\label{eq:matrix-units}
\end{equation}
Consequently, their span is a copy of the matrix algebra $M_m(\C)$, $\Pi_m^2=\Pi_m$, $Q_m^2=Q_m$, and
\begin{equation*}
C^{m-1}[0,T]=\mathcal P_{m-1}\oplus\mathcal N_m.
\end{equation*}
More explicitly, every $u\in C^{m-1}[0,T]$ has the unique decomposition
\begin{equation*}
u=P_{\mathbf d}+y,
\qquad
P_{\mathbf d}=\Pi_m u=\sum_{k=0}^{m-1}d_ke_k,
\qquad
d_k=\varepsilon_k u,
\qquad
y=Q_mu\in\mathcal N_m.
\end{equation*}
\end{proposition}

\begin{proof}
Since $\varepsilon_j(e_k)=\delta_{jk}$, direct composition gives \eqref{eq:matrix-units}. The identities for $\Pi_m$ and $Q_m$ follow. Moreover, $\varepsilon_k(Q_mu)=0$ for $0\le k<m$, so $Q_mu\in\mathcal N_m$, while $\Pi_mu\in\mathcal P_{m-1}$. Their intersection is zero because a polynomial in $\mathcal P_{m-1}$ is determined by its first $m$ endpoint derivatives. This proves both the direct sum and uniqueness.
\end{proof}

\begin{remark}[Role of the boundary block]
\Cref{prop:boundary-jets} supplies the finite-dimensional boundary coordinates absent from the flat core and couples them to the analytic factorization in \Cref{thm:boundary-factorization}. It leaves the homogeneous normal-form statement of \Cref{thm:normal-form} unchanged while providing the affine slices needed for prescribed initial data.
\end{remark}

Within $C^\infty[0,T]$, one has $\M=\bigcap_{m\ge1}\mathcal N_m$. This identity identifies the finite-normal-form calculus as the fully homogeneous block underlying every finite-jet problem. The matrix units restore the endpoint information suppressed on $\M$ without modifying \Cref{thm:normal-form}. For a fixed maximal order they provide a finite-dimensional algebraic interface; the analytic extension below attaches that interface to the completed smoothing sector through affine domains of the form $\mathcal P_{m-1}\oplus R(C[0,T])$.

\section{Completions and the boundary-extended causal-inverse framework}

We now pass from the finite-normal-form layer to an analytic one. The normal form of Section~2 remains a finite statement about $\A_\Gamma$, while the solution theory is formulated on natural ranges that may be larger than $\M$.  The next proposition shows that this change of domain does not discard the finite normal form: a monic normal-form operator first factors canonically on $\M$, and the same factorization has a compatible natural-range extension.  The completion that follows then enlarges only the negative-order causal convolution sector. The independently prescribed continuous kernels $K(t,s)$ in Sections~3.2--4 enter as external Volterra operators; composition with a smoothing causal inverse connects them to the preceding algebraic framework through the explicit reduced kernel \eqref{eq:GL}. Two different topologies are natural for the infinite smoothing and inverse expressions involved in this passage.

\begin{proposition}[From a finite normal form to its natural-range factorization]
\label{prop:core-range-bridge}
Let $\alpha\in\Gamma$, $\alpha>0$, and let
\begin{equation}
L_0=D^\alpha+\sum_{j=1}^{N}p_j(t)D^{\beta_j}+V_{K_0},
\qquad
p_j\in\Pcal,\qquad \beta_j\in\Gamma,\qquad \beta_j<\alpha,
\quad K_0\in\Kflat,
\label{eq:monic-normal-form}
\end{equation}
be a monic operator in the finite normal form of \Cref{thm:normal-form}.
Put $X=C[0,T]$, $R_\alpha=J^\alpha:X\to X$, and let
$A_\alpha=R_\alpha^{-1}$ on the natural range
$\D(A_\alpha)=J^\alpha X$.  Then:
\begin{enumerate}[label=\textup{(\roman*)},leftmargin=*]
\item
\begin{equation}
\M=J^\alpha\M\subset J^\alpha X,
\qquad
A_\alpha|_{\M}=D^\alpha.
\label{eq:core-range-compatibility}
\end{equation}
\item The operator
\begin{equation}
W_0=
\sum_{j=1}^{N}M_{p_j}J^{\alpha-\beta_j}
+V_{K_0}J^\alpha
\label{eq:algebraic-W0}
\end{equation}
extends from $\M$ to a bounded Volterra operator $V_{G_0}$ on $X$, with
\begin{equation}
G_0(t,s)=
\sum_{j=1}^{N}p_j(t)h_{\alpha-\beta_j}(t-s)
+\int_s^tK_0(t,\xi)h_\alpha(\xi-s)\dd\xi.
\label{eq:algebraic-reduced-kernel}
\end{equation}
\item On $\M$ one has the canonical right factorization
\begin{equation}
L_0=(I+W_0)D^\alpha.
\label{eq:algebraic-factorization}
\end{equation}
Consequently,
\begin{equation}
\widetilde L_0:=(I+V_{G_0})A_\alpha,
\qquad \D(\widetilde L_0)=J^\alpha X,
\label{eq:natural-range-extension}
\end{equation}
is a natural-range extension of $L_0$, in the precise sense that
$\widetilde L_0|_{\M}=L_0$.
\end{enumerate}
\end{proposition}

\begin{proof}
By \Cref{lem:flat-invariance}, $J^\alpha$ maps $\M$ bijectively onto
$\M$, with inverse $D^\alpha$.  This proves
\eqref{eq:core-range-compatibility}.  Since $\beta_j<\alpha$, the group law on
$\M$ gives
\begin{equation*}
D^{\beta_j}=J^{\alpha-\beta_j}D^\alpha,
\qquad
I=J^\alpha D^\alpha.
\end{equation*}
Hence
\begin{equation*}
p_jD^{\beta_j}=M_{p_j}J^{\alpha-\beta_j}D^\alpha,
\qquad
V_{K_0}=V_{K_0}J^\alpha D^\alpha,
\end{equation*}
and summation proves \eqref{eq:algebraic-factorization}.  The first term in
\eqref{eq:algebraic-W0} is the Volterra operator with kernel
$p_j(t)h_{\alpha-\beta_j}(t-s)$.  Fubini's theorem gives
\begin{equation*}
(V_{K_0}J^\alpha f)(t)
=\int_0^t\left[\int_s^tK_0(t,\xi)h_\alpha(\xi-s)\dd\xi\right]f(s)\dd s,
\end{equation*}
which proves \eqref{eq:algebraic-reduced-kernel}.  These kernels are
integrable on every finite Volterra triangle, so $V_{G_0}$ is bounded on
$X$.  Finally, \eqref{eq:core-range-compatibility} and
\eqref{eq:algebraic-factorization} give
$\widetilde L_0f=L_0f$ for $f\in\M$.
\end{proof}

No uniqueness among all possible operator extensions is asserted; $\widetilde L_0$ is the canonical natural-range extension induced by the right factorization \eqref{eq:algebraic-factorization} and the inverse $A_\alpha$.

\begin{remark}
The proposition identifies the exact element of the finite-normal-form
construction that enters the solution theory: the unique leading power
$D^\alpha$ selects the causal inverse $J^\alpha$, and every lower power becomes
a smoothing factor $J^{\alpha-\beta_j}$ after right factorization.  The general
theory below enlarges this algebraically generated class in three controlled
ways: it completes the negative-order convolution sector, allows continuous
perturbation coefficients, and admits an independent continuous bivariate
Volterra kernel. The monicity hypothesis specifies the scope of this direct
bridge. If the leading term of a general
normal form is $M_aD^\alpha$, a nowhere-vanishing $a$ can be divided out at the
equation level, but $1/a$ need not be polynomial and that normalization need
not remain inside $\A_\Gamma$.  If $a$ vanishes, no global reduction of this
kind is available.  Thus \Cref{thm:normal-form} applies to all elements of the
stated algebra, whereas the direct natural-range bridge established here is
asserted only for the monic subclass.
\end{remark}

\subsection{Kernel-norm completion and Prabhakar integrals}

Let
\begin{equation*}
\mathcal S_{\Gamma,T}^{\mathrm{fin}}
=\Span\{J^\beta:\beta\in\Gamma,\ \beta>0\}.
\end{equation*}
Using $J^\beta=V_{h_\beta}$, associate to $S=\sum_{j=1}^Nc_jJ^{\beta_j}$ the kernel $k_S=\sum_jc_jh_{\beta_j}$ and set
\begin{equation*}
\|S\|_{\mathrm{ker}}=\|k_S\|_{L^1(0,T)}.
\end{equation*}

\begin{proposition}[Smoothing completion]\label{prop:smoothing}
The completion of $\mathcal S_{\Gamma,T}^{\mathrm{fin}}$ in the kernel norm is a commutative Banach convolution algebra. Because $\Z\subset\Gamma$, its kernel completion is $L^1(0,T)$; after adjoining the identity it is identified with
\begin{equation*}
\C I\oplus\{V_k:k\in L^1(0,T)\}.
\end{equation*}
The unitalized algebra is equipped with the norm
\begin{equation*}
\|cI+V_k\|_{\mathrm{un}}:=|c|+\|k\|_{L^1(0,T)}.
\end{equation*}
Moreover,
\begin{equation*}
\|V_k\|_{C[0,T]\to C[0,T]}\le\|k\|_{L^1(0,T)}.
\end{equation*}
\end{proposition}

\begin{proof}
The map $S\mapsto k_S$ identifies the finite operator span injectively with its finite kernel span: if $k_S=0$ in $L^1(0,T)$, then the corresponding generalized-power combination vanishes on $(0,T)$, and separation of the distinct powers gives $S=0$. Thus $\|S\|_{\mathrm{ker}}=\|k_S\|_1$ is a norm, and the completion in question is precisely the completion of this kernel realization. The identity $h_\alpha*h_\beta=h_{\alpha+\beta}$ makes the finite kernel span a convolution algebra, and Young's inequality gives $\|k*\ell\|_1\le\|k\|_1\|\ell\|_1$. Because $\Z\subset\Gamma$, all $J^n$, $n\ge1$, occur, with kernels $t^{n-1}/(n-1)!$; hence their linear span is the set of polynomial kernels. Density of the polynomials in $L^1(0,T)$ proves that the kernel-norm completion is $L^1(0,T)$.
For the unitalization, the identity
\begin{equation*}
(cI+V_k)(dI+V_\ell)=cdI+V_{c\ell+dk+k*\ell}
\end{equation*}
and Young's inequality show that $\|\cdot\|_{\mathrm{un}}$ is submultiplicative and complete. Finally,
\begin{equation*}
|(V_kf)(t)|\le\|f\|_\infty\int_0^t|k(t-s)|\dd s
\end{equation*}
gives the operator bound.
\end{proof}

\begin{remark}
Proposition~\ref{prop:core-range-bridge} gives the algebraic and domain-theoretic bridge from a finite normal form to the analytic factorization.  The preceding smoothing-completion proposition performs the next, distinct step: it identifies the completed convolution sector on which more general causal inverses are available.  The negative-order convolution sector consists of finite sums, but its kernel-norm closure is $L^1(0,T)$; hence every admissible smoothing inverse $R=V_r$ with $r\in L^1(0,T)$, including every integrable Sonine kernel, belongs to this completed sector. The causal theory below supplements this membership by injectivity and quantitative weak-singularity bounds and, for the Prabhakar family, by the explicit norm-convergent expansion in the next proposition. An independently prescribed continuous bivariate operator $V_K$ is not absorbed into this completion: only after composition with $R$ does it enter the single reduced kernel \eqref{eq:GL}. This is the mechanism that extends the finite-normal-form factorization to the nonconvolution equations of Sections~3.2--4.
The density argument identifies the precise analytic closure through which the finite algebraic smoothing sector reaches the Sonine and Prabhakar kernels used later.
\end{remark}

For $\rho,\mu>0$ and $\gamma,z\in\C$, define the three-parameter Mittag--Leffler function
\begin{equation*}
E_{\rho,\mu}^{\gamma}(z)
=\sum_{n=0}^{\infty}\frac{(\gamma)_n}{n!\,\Gamma(\rho n+\mu)}z^n
\end{equation*}
and the Prabhakar integral
\begin{equation}
(\E_{\rho,\mu,\omega}^{\gamma}f)(t)
=\int_0^t(t-s)^{\mu-1}E_{\rho,\mu}^{\gamma}\bigl(\omega(t-s)^\rho\bigr)f(s)\dd s.
\end{equation}
Here the roman $E_{\rho,\mu}^{\gamma}$ denotes the Mittag--Leffler function, whereas the calligraphic $\E_{\rho,\mu,\omega}^{\gamma}$ denotes the associated convolution operator.

\begin{proposition}[Prabhakar expansion]\label{prop:prabhakar-expansion}
Let $\rho,\mu>0$ and $\gamma,\omega\in\C$. If $\Gamma$ contains $\rho$ and $\mu$, then $\E_{\rho,\mu,\omega}^{\gamma}$ belongs to the kernel-norm completion of the smoothing sector $\mathcal S_{\Gamma,T}^{\mathrm{fin}}$ associated with $\A_\Gamma$, and
\begin{equation}
\E_{\rho,\mu,\omega}^{\gamma}
=\sum_{n=0}^{\infty}\frac{(\gamma)_n\omega^n}{n!}J^{\mu+\rho n},
\label{eq:prabhakar-expansion}
\end{equation}
with absolute convergence in the kernel norm on every finite interval.
\end{proposition}

\begin{proof}
Substitution of the Mittag--Leffler series into the integral gives \eqref{eq:prabhakar-expansion} formally. Its kernel-norm majorant is
\begin{equation*}
\sum_{n=0}^{\infty}\frac{|(\gamma)_n|\,|\omega|^n}{n!}
\frac{T^{\mu+\rho n}}{\Gamma(\mu+\rho n+1)}.
\end{equation*}
The quotient $(\gamma)_n/n!$ grows at most polynomially, whereas the gamma factor in the denominator dominates every fixed exponential. Hence the series converges and termwise integration is justified in $L^1(0,T)$.
\end{proof}

Write
\begin{equation}
r_{\rho,\mu,\omega}^{\gamma}(t)
=t^{\mu-1}E_{\rho,\mu}^{\gamma}(\omega t^\rho).
\label{eq:prabhakar-kernel}
\end{equation}
On every finite interval,
\begin{equation*}
|r_{\rho,\mu,\omega}^{\gamma}(t)|\le C_Tt^{\mu-1},
\qquad 0<t\le T,
\end{equation*}
and the defining series shows
\begin{equation*}
r_{\rho,\mu,\omega}^{\gamma}(t)
=\frac{t^{\mu-1}}{\Gamma(\mu)}\bigl(1+O(t^\rho)\bigr),
\qquad t\downarrow0.
\end{equation*}
Consequently, \Cref{lem:injectivity} proves directly that $\E_{\rho,\mu,\omega}^{\gamma}$ is injective on $C[0,T]$. No regularized derivative is applied outside its classical domain.

\subsection{Admissible causal inverses and factorization}

Set $X=C[0,T]$. The perturbation argument needs only an injective smoothing inverse and power-type control of its compositions with the lower operators.
From this point onward, the coefficients $p_j\in C[0,T]$ are analytic perturbation coefficients. Unless they are polynomials, they are not asserted to be multipliers in the finite normal-form algebra $\A_\Gamma$; their role belongs solely to the causal-factorization layer.
When $R=J^\alpha$, $B_j=D^{\beta_j}$, the $p_j$ are polynomial, and the
Volterra kernel is the normal-form remainder $K_0$, the construction below
restricts on $\M$ to \eqref{eq:algebraic-factorization}.  Thus
\Cref{prop:factorization} is an analytic generalization of the bridge
proposition, not an unrelated second factorization.

\begin{definition}
An admissible causal inverse is an injective convolution operator
\begin{equation*}
R=V_r:X\to X,
\qquad r\in L^1(0,T),
\end{equation*}
such that, for some $C_0>0$ and $\eta_0>0$,
\begin{equation*}
|r(t)|\le C_0t^{\eta_0-1}
\qquad\text{for a.e. }0<t\le T.
\end{equation*}
Its inverse is denoted by $A=R^{-1}$ on the natural domain
\begin{equation*}
\D(A)=R(X).
\end{equation*}
\end{definition}

Let $B_1,\ldots,B_N$ be linear operators on $\D(A)$ satisfying
\begin{equation}
B_jR=V_{q_j},
\qquad q_j\in L^1(0,T),
\qquad |q_j(t)|\le C_jt^{\eta_j-1}\ \text{a.e.},
\quad \eta_j>0.
\label{eq:BjR}
\end{equation}
For $p_j\in C[0,T]$ and $K\in C(\Delta_T)$, define
\begin{equation}
Lu=Au+\sum_{j=1}^{N}p_j(t)B_ju(t)+\int_0^tK(t,s)u(s)\dd s.
\label{eq:L}
\end{equation}
For $v\in X$, let $W_L=V_{G_L}$, where
\begin{equation}
G_L(t,s)=\sum_{j=1}^{N}p_j(t)q_j(t-s)
+\int_s^tK(t,\xi)r(\xi-s)\dd\xi.
\label{eq:GL}
\end{equation}
The kernels are understood through measurable representatives. Young's inequality shows that every $V_{q_j}$ maps $X$ continuously into itself. The integral part of $V_{G_L}$ is the composition $V_KR$, which also maps $X$ into itself because $K\in C(\Delta_T)$. Hence $W_L=V_{G_L}$ is a bounded operator $X\to X$, even though $G_L$ need not extend continuously to the diagonal.

\begin{proposition}[Causal factorization]\label{prop:factorization}
On $\D(A)$,
\begin{equation}
L=(I+W_L)A.
\label{eq:factorization}
\end{equation}
\end{proposition}

\begin{proof}
Write $u=Rv$, so that $Au=v$. By \eqref{eq:BjR}, $p_jB_ju=p_jV_{q_j}v$. Fubini's theorem gives
\begin{equation*}
\int_0^tK(t,s)(Rv)(s)\dd s
=\int_0^t\left[\int_\tau^tK(t,s)r(s-\tau)\dd s\right]v(\tau)\dd\tau.
\end{equation*}
The bracketed expression is the second term of \eqref{eq:GL}, which proves \eqref{eq:factorization}.
\end{proof}

\subsection{Weakly singular resolvents and the general perturbation theorem}

Put
\begin{equation*}
\eta=\min\bigl(\{1\}\cup\{\eta_j:1\le j\le N\}\bigr)>0.
\end{equation*}
The integral term in \eqref{eq:GL} is $O((t-s)^{\eta_0})$. Hence, after increasing a constant if necessary,
\begin{equation*}
|G_L(t,s)|\le C_L(t-s)^{\eta-1}
\qquad\text{for a.e. }0\le s<t\le T.
\end{equation*}

We use the standard weakly singular Volterra resolvent construction; see, for example, \cite{Becker2011,Brunner2017,Gripenberg1990,Pruss1993}.

\begin{lemma}[Weakly singular resolvent]\label{lem:resolvent}
Let $G$ be measurable on the open Volterra triangle, suppose that $V_G$ maps $C[0,T]$ into itself, and assume that, for some $C_G>0$ and $\eta>0$,
\begin{equation*}
|G(t,s)|\le C_G(t-s)^{\eta-1}
\qquad\text{for a.e. }0\le s<t\le T.
\end{equation*}
Define $G_1=G$ and
\begin{equation*}
G_{n+1}(t,s)=\int_s^tG(t,\xi)G_n(\xi,s)\dd\xi.
\end{equation*}
Then
\begin{equation*}
|G_n(t,s)|\le\frac{(C_G\Gamma(\eta))^n}{\Gamma(n\eta)}(t-s)^{n\eta-1}.
\end{equation*}
Consequently, the kernel series converges absolutely for almost every $0\le s<t\le T$, and the induced operator series converges in operator norm on $C[0,T]$. In this precise sense,
\begin{equation}
R_G(t,s)=\sum_{n=1}^{\infty}(-1)^{n-1}G_n(t,s)
\end{equation}
defines the Volterra resolvent kernel and
\begin{equation}
(I+V_G)^{-1}=I-V_{R_G}
\end{equation}
on $C[0,T]$.
\end{lemma}

\begin{proof}
Induction and the beta integral give the pointwise estimate. In addition,
\begin{equation*}
\|V_{G_n}\|_{C\to C}
\le\frac{(C_G\Gamma(\eta)T^\eta)^n}{\Gamma(n\eta+1)}.
\end{equation*}
For each fixed $t-s>0$, the pointwise majorant is summable, so the kernel series is absolutely convergent almost everywhere off the diagonal. The denominator in the operator estimate dominates every geometric sequence; hence the induced Neumann series converges in operator norm for each finite $T$. The same beta-function majorants justify Tonelli--Fubini in every finite iterate, giving $V_G^n=V_{G_n}$. Therefore the operator-norm sum of the Neumann series is represented by the almost-everywhere kernel sum above. Multiplication of the convergent operator series by $I+V_G$ then yields $(I+V_G)(I-V_{R_G})=(I-V_{R_G})(I+V_G)=I$, with the sign convention used in the definition of $R_G$.
\end{proof}

\begin{theorem}[Causal-inverse Volterra perturbation]\label{thm:causal-inverse}
Let $R,A,B_j,p_j$ and $K$ satisfy the preceding assumptions. For every $f\in C[0,T]$, the equation
\begin{equation*}
Lu=f
\end{equation*}
has a unique solution $u\in\D(A)$. If $R_L$ is the resolvent kernel generated by \eqref{eq:GL}, then
\begin{equation}
u=R(I-V_{R_L})f.
\label{eq:general-solution}
\end{equation}
Consequently, the classical Volterra resolvent mechanism requires no smallness condition on $p_j$, $K$, or the interval length.
\end{theorem}

\begin{proof}
With $u=Rv$, \Cref{prop:factorization} transforms the equation into $(I+W_L)v=f$. \Cref{lem:resolvent} gives the unique solution $v=(I-V_{R_L})f$. Applying $R$ proves \eqref{eq:general-solution}. Conversely, every $u\in\D(A)$ is $Rv$ for a unique $v$ because $R$ is injective, so uniqueness follows from the second-kind Volterra equation.
\end{proof}

The same factorization admits a finite boundary block. Let $\mathcal B\subset X$ be finite-dimensional and assume
\begin{equation}
\mathcal B\cap R(X)=\{0\}.
\label{eq:boundary-transversality}
\end{equation}
For linear maps $\Lambda_j:\mathcal B\to X$, define operators on the direct sum
\begin{equation*}
\D(A_\partial)=\mathcal B\oplus R(X)
\end{equation*}
by
\begin{equation}
A_\partial(P+Rv)=v,
\qquad
B_{j,\partial}(P+Rv)=\Lambda_jP+V_{q_j}v.
\label{eq:boundary-extended-operators}
\end{equation}
Thus $A_\partial$ annihilates the boundary block, while each lower operator has a prescribed finite-dimensional action there. Put
\begin{equation}
F_\partial(P)(t)
=\sum_{j=1}^{N}p_j(t)(\Lambda_jP)(t)
+\int_0^tK(t,s)P(s)\dd s.
\label{eq:boundary-forcing}
\end{equation}

\begin{theorem}[Affine boundary-jet factorization]\label{thm:boundary-factorization}
Under the hypotheses of \Cref{thm:causal-inverse} and \eqref{eq:boundary-transversality}, fix $P\in\mathcal B$ and define
\begin{equation*}
L_\partial u
=A_\partial u+\sum_{j=1}^{N}p_j(t)B_{j,\partial}u+V_Ku.
\end{equation*}
For every $f\in X$, the equation $L_\partial u=f$ has a unique solution in the affine slice $P+R(X)$, namely
\begin{equation}
u=P+R(I-V_{R_L})\bigl(f-F_\partial(P)\bigr),
\label{eq:boundary-solution}
\end{equation}
where $R_L$ is generated by the same reduced kernel \eqref{eq:GL} as in the homogeneous problem. Equivalently, writing $u=P+Rv$ gives the block factorization
\begin{equation}
L_\partial(P+Rv)=F_\partial(P)+(I+W_L)v.
\label{eq:boundary-block-factorization}
\end{equation}
\end{theorem}

\begin{proof}
The definitions and \Cref{prop:factorization} give \eqref{eq:boundary-block-factorization}. Hence $L_\partial u=f$ is equivalent to
\begin{equation*}
(I+W_L)v=f-F_\partial(P).
\end{equation*}
The resolvent lemma gives the unique density $v=(I-V_{R_L})(f-F_\partial(P))$, and injectivity of $R$ together with \eqref{eq:boundary-transversality} gives uniqueness in $P+R(X)$.
\end{proof}

When $\mathcal B=\mathcal P_{m-1}$ and $R(X)\subset\mathcal N_m$, the projectors of \eqref{eq:jet-projectors} restrict to
\begin{equation*}
\Pi_m(P+Rv)=P,
\qquad
Q_m(P+Rv)=Rv.
\end{equation*}
Consequently, the extended operators and the full equation operator have the algebraic block forms
\begin{align}
A_\partial&=AQ_m,
\qquad B_{j,\partial}=\Lambda_j\Pi_m+B_jQ_m,
\label{eq:boundary-operator-blocks}\\
L_\partial
&=\left(A+\sum_{j=1}^{N}p_jB_j+V_K\right)Q_m
+\left(\sum_{j=1}^{N}p_j\Lambda_j+V_K\right)\Pi_m.
\label{eq:augmented-operator-form}
\end{align}
Thus the homogeneous normal-form sector and the finite jet matrix sector enter a single block-operator formula; the prescribed data select an affine fiber of its boundary projection.

\begin{remark}[Scope of the perturbation theorem]
\Cref{thm:causal-inverse,thm:boundary-factorization} combine the classical weakly singular Volterra resolvent mechanism of \Cref{lem:resolvent} with the operator factorizations \eqref{eq:factorization} and \eqref{eq:boundary-block-factorization}. The homogeneous and boundary blocks share the single second-kind kernel \eqref{eq:GL}; prescribed finite-dimensional data alter only the explicit forcing \eqref{eq:boundary-forcing}.
The structural output of the present framework is the causal factorization and the explicit reduced kernel that retains the independently prescribed bivariate perturbation; classical Volterra theory then supplies the resolvent of that reduced kernel.
\end{remark}

\begin{remark}
\Cref{thm:causal-inverse} is a natural-range statement, while \Cref{thm:boundary-factorization} is unconditional on each prescribed affine slice. Classical fractional derivatives may still have smaller regularity domains. The exact relation between the affine natural domain and the classical regularized-Prabhakar domain is established in \Cref{prop:classical-range}.
\end{remark}

\section{Fractional and boundary-jet realizations of the causal-inverse calculus}

\subsection{Finite fractional leading order}

Let $\alpha>0$. Take $R=J^\alpha$ and $A=R^{-1}$ on the natural range $\D(A)=J^\alpha C[0,T]$. To distinguish operators defined by range factorization from classical fractional derivatives, write
\begin{equation*}
\mathfrak D_{\alpha,\mathrm{nr}}^\alpha(J^\alpha v):=v,
\qquad
\mathfrak D_{\alpha,\mathrm{nr}}^\beta(J^\alpha v):=J^{\alpha-\beta}v,
\qquad \beta<\alpha.
\end{equation*}
Thus $A=\mathfrak D_{\alpha,\mathrm{nr}}^\alpha$. Injectivity of $J^\alpha$ makes these definitions unambiguous. We reserve the undecorated symbol $D^\beta$ for the flat-core operator, or for the classical Riemann--Liouville derivative on its classical domain. On every point of intersection with that domain, $\mathfrak D_{\alpha,\mathrm{nr}}^\beta$ agrees with $D^\beta$, and
\begin{equation}
\mathfrak D_{\alpha,\mathrm{nr}}^\beta J^\alpha
=J^{\alpha-\beta}=V_{h_{\alpha-\beta}}.
\label{eq:lower-power}
\end{equation}
By \eqref{eq:core-range-compatibility},
\begin{equation*}
\M=J^\alpha\M\subset J^\alpha C[0,T],
\qquad
\mathfrak D_{\alpha,\mathrm{nr}}^\beta|_{\M}=D^\beta
\quad(\beta\le\alpha).
\end{equation*}
Thus the natural-range operator is a compatible extension of the flat-core
operator that supplies the leading term of the finite normal form.  It is not
being defined on all of $C^\infty[0,T]$, but precisely on the larger natural
range $J^\alpha C[0,T]$.  Likewise, \eqref{eq:lower-power} is the
natural-range extension of the algebraic identity
$D^\beta=J^{\alpha-\beta}D^\alpha$ on $\M$.

\begin{corollary}[Fractional Volterra perturbation]\label{cor:fractional}
Let $\alpha>0$. For $p_j\in C[0,T]$, $\beta_j<\alpha$, and $K\in C(\Delta_T)$, the equation
\begin{equation}
\mathfrak D_{\alpha,\mathrm{nr}}^\alpha u
+\sum_{j=1}^{N}p_j(t)\mathfrak D_{\alpha,\mathrm{nr}}^{\beta_j}u
+\int_0^tK(t,s)u(s)\dd s=f(t)
\label{eq:fractional-equation}
\end{equation}
has a unique solution in $J^\alpha C[0,T]$. Its reduced kernel is
\begin{equation}
G_\alpha(t,s)
=\sum_{j=1}^{N}\frac{p_j(t)}{\Gamma(\alpha-\beta_j)}(t-s)^{\alpha-\beta_j-1}
+\frac1{\Gamma(\alpha)}\int_s^tK(t,\xi)(\xi-s)^{\alpha-1}\dd\xi,
\end{equation}
and
\begin{equation}
u=J^\alpha(I-V_{R_\alpha})f.
\end{equation}
\end{corollary}

The finite jet layer gives the corresponding nonzero-data statement without changing the reduced kernel. Let $m-1<\alpha\le m$. Since
\begin{equation*}
J^\alpha=J^{m-1}J^{\alpha-m+1},
\end{equation*}
every element of $J^\alpha X$ belongs to $C^{m-1}[0,T]$ and has zero jets through order $m-1$. Hence
\begin{equation*}
\mathcal D_{\alpha,m}
=\mathcal P_{m-1}\oplus J^\alpha X
\end{equation*}
is a direct sum. On this domain define the boundary-extended operators
\begin{align}
{}^C\!\mathfrak D_m^\alpha(P+J^\alpha v)&=v,\label{eq:extended-caputo-leading}\\
{}^C\!\mathfrak D_m^\beta(P+J^\alpha v)
&={}^CD^\beta P+J^{\alpha-\beta}v,
\qquad 0\le\beta<\alpha.
\label{eq:extended-caputo-lower}
\end{align}
The leading operator annihilates $\mathcal P_{m-1}$, and the lower operators retain its explicit Caputo images.

\begin{corollary}[Fractional perturbation with prescribed jets]\label{cor:fractional-jets}
Let $m-1<\alpha\le m$, $0\le\beta_j<\alpha$, and prescribe $d_0,\ldots,d_{m-1}\in\C$. Put $P_{\mathbf d}=\sum_{k=0}^{m-1}d_ke_k$ and
\begin{equation}
F_{\mathbf d}(t)
=f(t)-\sum_{j=1}^{N}p_j(t){}^CD^{\beta_j}P_{\mathbf d}(t)
-\int_0^tK(t,s)P_{\mathbf d}(s)\dd s.
\label{eq:fractional-boundary-forcing}
\end{equation}
Then
\begin{equation}
{}^C\!\mathfrak D_m^\alpha u
+\sum_{j=1}^{N}p_j(t){}^C\!\mathfrak D_m^{\beta_j}u
+\int_0^tK(t,s)u(s)\dd s=f(t),
\qquad \varepsilon_ku=d_k,\quad 0\le k<m,
\label{eq:fractional-boundary-equation}
\end{equation}
has a unique solution in the affine slice $P_{\mathbf d}+J^\alpha X$, given by
\begin{equation}
u=P_{\mathbf d}+J^\alpha(I-V_{R_\alpha})F_{\mathbf d},
\label{eq:fractional-boundary-solution}
\end{equation}
where $R_\alpha$ is generated by the same kernel $G_\alpha$ as in \Cref{cor:fractional}.
\end{corollary}

\begin{proof}
Apply \Cref{thm:boundary-factorization} with $R=J^\alpha$, $\mathcal B=\mathcal P_{m-1}$ and $\Lambda_jP={}^CD^{\beta_j}P$. The zero-jet property of $J^\alpha X$ identifies the boundary coordinates with $\varepsilon_0,\ldots,\varepsilon_{m-1}$, and \eqref{eq:lower-power} leaves the reduced kernel unchanged.
\end{proof}

\subsection{Admissible Sonine-kernel Volterra perturbations}

Let $\kappa,k\in L^1(0,T)$ be a Sonine pair, meaning that
\begin{equation}
(\kappa*k)(t)=1,
\qquad 0<t\le T.
\label{eq:sonine}
\end{equation}
Assume in addition that
\begin{equation}
|\kappa(t)|\le C_\kappa t^{\eta_\kappa-1},
\qquad \eta_\kappa>0,
\label{eq:kappa-bound}
\end{equation}
and that $\kappa$ does not vanish almost everywhere on any interval $(0,\varepsilon)$. Then $I_\kappa:=V_\kappa$ is injective by \Cref{lem:injectivity} and is an admissible causal inverse. On its natural range define
\begin{equation*}
D_ku:=\frac{\dd}{\dd t}V_ku,
\qquad \D(D_k)=I_\kappa(C[0,T]).
\end{equation*}
Indeed, if $u=I_\kappa v$, associativity and \eqref{eq:sonine} give
\begin{equation*}
V_ku=V_{k*\kappa}v=Jv,
\qquad D_ku=v.
\end{equation*}
Thus $D_k=I_\kappa^{-1}$ on the natural range. On every classical domain where differentiation of the convolution is valid, this is the Riemann--Liouville-type general fractional derivative generated by $k$; if $u(0)=0$, it also agrees with the corresponding Caputo-type derivative $V_ku'$ \cite{Luchko2022}.

Lower general fractional derivatives can be retained without imposing a common power law. Let $k_j\in L^1(0,T)$ be such that
\begin{equation}
g_j:=k_j*\kappa\in\AC[0,T],
\qquad g_j(0)=0,
\qquad q_j:=g_j'\ \text{a.e.},
\qquad |q_j(t)|\le C_jt^{\eta_j-1}\ \text{a.e.}
\label{eq:lower-sonine}
\end{equation}
for some $\eta_j>0$. Thus $q_j\in L^1(0,T)$. If $u=I_\kappa v$ with $v\in C[0,T]$, associativity gives $V_{k_j}u=V_{g_j}v$. The standard differentiation formula for an absolutely continuous convolution kernel, together with $g_j(0)=0$, yields
\begin{equation}
(V_{k_j}I_\kappa v)'=(V_{g_j}v)'=V_{q_j}v.
\label{eq:lower-sonine-composition}
\end{equation}
Although $q_j=g_j'$ is initially defined only almost everywhere, $V_{q_j}v$ has a unique continuous representative. Equation \eqref{eq:lower-sonine-composition} therefore shows that $V_{k_j}I_\kappa v\in C^1[0,T]$ and defines $D_{k_j}u=(V_{k_j}u)'$ as that continuous function. Hence all terms in the equation below are equalities in $C[0,T]$, not merely almost-everywhere identities.

\begin{theorem}[Sonine--Volterra perturbation]\label{thm:sonine-volterra}
Suppose that \eqref{eq:sonine}--\eqref{eq:lower-sonine} hold, $p_0,p_j\in C[0,T]$, and $K\in C(\Delta_T)$. Then, for every $f\in C[0,T]$, the equation
\begin{equation}
D_ku(t)+p_0(t)u(t)+\sum_{j=1}^{N}p_j(t)D_{k_j}u(t)
+\int_0^tK(t,s)u(s)\dd s=f(t)
\label{eq:sonine-volterra}
\end{equation}
has a unique solution in $I_\kappa(C[0,T])$. It is
\begin{equation}
u=I_\kappa(I-V_{R_S})f,
\end{equation}
where $R_S$ is generated by
\begin{equation}
G_S(t,s)=p_0(t)\kappa(t-s)
+\sum_{j=1}^{N}p_j(t)q_j(t-s)
+\int_s^tK(t,\xi)\kappa(\xi-s)\dd\xi.
\label{eq:GS}
\end{equation}
Beyond the continuity assumption $K\in C(\Delta_T)$, no convolution structure on $K$ is required, and no norm-smallness condition on the perturbation coefficients or on $K$ is imposed.
\end{theorem}

\begin{proof}
Take $R=I_\kappa$, $A=D_k$, include the identity among the lower operators with kernel $q_0=\kappa$, and use \eqref{eq:lower-sonine-composition} for the remaining terms. Conditions \eqref{eq:kappa-bound} and \eqref{eq:lower-sonine} verify the hypotheses of \Cref{thm:causal-inverse}; its factorization gives \eqref{eq:GS} and its resolvent formula gives the stated solution.
\end{proof}

In relation to the operational calculus of Luchko \cite{Luchko2021}, the later construction and inversion theory for Sonine kernels \cite{Luchko2022,Ortigueira2024}, and the continuous-variable-coefficient Sonine equations of \cite{Guzoglu2026}, \Cref{thm:sonine-volterra} adds an independently prescribed continuous bivariate kernel to the variable-coefficient setting. This kernel survives the causal reduction explicitly as the last term of \eqref{eq:GS}, simultaneously with compatible lower Sonine derivatives, and produces a genuinely nonconvolution reduced kernel.

\begin{example}[A genuinely nonconvolution Sonine perturbation]
Let $0<\alpha<1$, $\kappa=h_\alpha$, and $k=h_{1-\alpha}$. Then $k*\kappa=1$, $I_\kappa=J^\alpha$, and $D_k=\mathfrak D_{\alpha,\mathrm{nr}}^\alpha$ on $J^\alpha C[0,T]$ (and hence agrees there with the classical $D^\alpha$ wherever the latter is defined). Take $N=0$, $p_0=0$, and $K(t,s)=t+s$. Equation \eqref{eq:sonine-volterra} becomes
\begin{equation*}
\mathfrak D_{\alpha,\mathrm{nr}}^\alpha u(t)
+\int_0^t(t+s)u(s)\dd s=f(t),
\qquad u\in J^\alpha C[0,T].
\end{equation*}
The reduced kernel is explicitly
\begin{align*}
G_{\mathrm{nc}}(t,s)
&=\frac1{\Gamma(\alpha)}\int_s^t(t+\xi)(\xi-s)^{\alpha-1}\dd\xi\\
&=\frac{(t+s)(t-s)^\alpha}{\Gamma(\alpha+1)}
+\frac{\alpha(t-s)^{\alpha+1}}{\Gamma(\alpha+2)}.
\end{align*}
The factor $t+s$ shows directly that $G_{\mathrm{nc}}$ is not a convolution kernel. If $R_{\mathrm{nc}}$ is its Volterra resolvent, the unique natural-range solution is
\begin{equation*}
u=J^\alpha(I-V_{R_{\mathrm{nc}}})f.
\end{equation*}
Thus the independently bivariate term is not merely appended formally: it remains visible after causal reduction and is handled without a convolution assumption.
\end{example}

\subsection{Distributed-order realization}

Distributed-order derivatives and their inverse kernels have an established theory; see, in particular, Kochubei \cite{Kochubei2008}. The next proposition verifies, for every nonzero finite positive Borel measure supported in a compact subinterval of $(0,1)$, the precise positivity, injectivity, and pointwise weak-singularity hypotheses required by \Cref{thm:sonine-volterra}. It thereby places continuous-order equations with an independently prescribed continuous bivariate Volterra perturbation inside the causal-factorization framework.

Let $\nu$ be a nonzero finite positive Borel measure supported in a compact interval $[a,b]\subset(0,1)$ and put
\begin{equation}
k_\nu(t)=\int_{(0,1)}h_{1-\alpha}(t)\dd\nu(\alpha).
\label{eq:distributed-k}
\end{equation}

\begin{proposition}[Distributed-order Sonine associate]\label{prop:distributed}
The kernel $k_\nu$ has a completely monotone Sonine associate $\kappa_\nu\in L^1_{\mathrm{loc}}(0,\infty)$ satisfying
\begin{equation}
\widehat\kappa_\nu(z)=\left(\int_{[a,b]}z^\alpha\dd\nu(\alpha)\right)^{-1},
\qquad z>0.
\label{eq:kappahat}
\end{equation}
Moreover, $\kappa_\nu(t)>0$ for $t>0$ and, for every $T>0$,
\begin{equation}
0<\kappa_\nu(t)\le C_Tt^{a-1},
\qquad 0<t\le T.
\label{eq:kappa-bound-distributed}
\end{equation}
Thus $\kappa_\nu$ satisfies all admissibility and injectivity hypotheses of \Cref{thm:sonine-volterra}.
\end{proposition}

\begin{proof}
Set $\Phi(z)=\int_{[a,b]}z^\alpha\dd\nu(\alpha)$. Each $z^\alpha$, $0<\alpha<1$, is a complete Bernstein function, and positive finite sums preserve this class. To approximate the integral, partition $[a,b]$ into subintervals $I_{n,k}$ of mesh tending to zero, choose $\alpha_{n,k}\in I_{n,k}$ whenever $\nu(I_{n,k})>0$, and set
\begin{equation*}
\Phi_n(z)=\sum_k\nu(I_{n,k})z^{\alpha_{n,k}}.
\end{equation*}
Continuity of $(\alpha,z)\mapsto z^\alpha$ on $[a,b]\times K$ shows that $\Phi_n\to\Phi$ uniformly for every compact $K\subset(0,\infty)$. Thus the closure property in \cite[Corollary~7.6]{Schilling2012} shows that $\Phi$ is again a complete Bernstein function. The reciprocal characterization \cite[Theorem~7.3]{Schilling2012} therefore implies that $1/\Phi$ is a Stieltjes function. If $M=\nu([a,b])>0$, then for $z\ge1$ one has $\Phi(z)\ge Mz^a$, and for $0<z\le1$ one has $\Phi(z)\ge Mz^b$. Consequently,
\begin{equation*}
0\le\frac1{\Phi(z)}\le\frac1M z^{-a}\longrightarrow0
\quad(z\to\infty),
\qquad
0\le\frac z{\Phi(z)}\le\frac1M z^{1-b}\longrightarrow0
\quad(z\downarrow0).
\end{equation*}
Thus the Stieltjes representation in \cite[Theorem~2.2]{Schilling2012} has neither a constant term nor a $z^{-1}$ term. Hence there is a nonzero positive measure $\sigma$ on $[0,\infty)$ such that
\begin{equation}
\frac1{\Phi(z)}=\int_{[0,\infty)}\frac1{z+r}\dd\sigma(r)
=\int_0^\infty e^{-zt}\kappa_\nu(t)\dd t,
\qquad
\kappa_\nu(t):=\int_{[0,\infty)}e^{-rt}\dd\sigma(r).
\label{eq:stieltjes-kappa}
\end{equation}
Tonelli's theorem and the Stieltjes integrability condition give
\begin{equation*}
\int_0^T\kappa_\nu(t)\dd t
=\int_{[0,\infty)}\frac{1-e^{-rT}}{r}\dd\sigma(r)<\infty,
\end{equation*}
where the integrand is interpreted as $T$ at $r=0$. Thus $\kappa_\nu\in L^1_{\mathrm{loc}}(0,\infty)$; differentiating under the positive integral in \eqref{eq:stieltjes-kappa} shows that it is completely monotone, and $\sigma\ne0$ gives $\kappa_\nu(t)>0$ for every $t>0$. Also
\begin{equation*}
\widehat k_\nu(z)=\frac{\Phi(z)}{z},
\qquad \widehat k_\nu(z)\widehat\kappa_\nu(z)=\frac1z.
\end{equation*}
Uniqueness of the Laplace transform yields $k_\nu*\kappa_\nu=1$.

It remains to verify the pointwise bound. For $0<t\le1$, monotonicity gives
\begin{equation*}
\widehat\kappa_\nu(1/t)
\ge\int_0^te^{-s/t}\kappa_\nu(s)\dd s
\ge e^{-1}t\kappa_\nu(t).
\end{equation*}
Because $\Phi(1/t)\ge\nu([a,b])t^{-a}$, \eqref{eq:kappahat} implies $\kappa_\nu(t)\le e\nu([a,b])^{-1}t^{a-1}$. On a remaining compact interval away from zero the bound follows from continuity.
\end{proof}

\begin{remark}[Status of the distributed-order step]
\Cref{prop:distributed} uses the complete-Bernstein/Stieltjes mechanism to show that a genuinely continuous positive order measure satisfies, without an extra assumed kernel estimate, the precise admissibility hypotheses needed to retain the external bivariate perturbation in \Cref{thm:sonine-volterra}.
\end{remark}

For $u\in\AC[0,T]$ with $u(0)=0$, Fubini's theorem gives
\begin{equation}
D_{k_\nu}u=\int_{(0,1)}{}^{C}\!D^\alpha u\dd\nu(\alpha).
\label{eq:distributed-caputo}
\end{equation}
Indeed, for $\alpha\in[a,b]$ the kernels $h_{1-\alpha}$ are dominated on $(0,T]$ by an integrable multiple of $t^{-a}+t^{-b}$. Since $u'\in L^1(0,T)$ and $\nu$ is finite, Tonelli--Fubini justifies interchanging the $\alpha$-integration with causal convolution; the identity $(k_\nu*u)'=k_\nu*u'$ follows from $u(0)=0$. This proves \eqref{eq:distributed-caputo} without an implicit interchange of a singular integral and differentiation.

Consequently, \Cref{thm:sonine-volterra} with $N=0$ gives a unique natural-range solution of
\begin{equation}
D_{k_\nu}u(t)+p(t)u(t)+\int_0^tK(t,s)u(s)\dd s=f(t),
\qquad u\in I_{\kappa_\nu}(C[0,T]),
\end{equation}
namely
\begin{equation}
u=I_{\kappa_\nu}(I-V_{R_\nu})f,
\end{equation}
where $R_\nu$ is generated by
\begin{equation}
G_\nu(t,s)=p(t)\kappa_\nu(t-s)
+\int_s^tK(t,\xi)\kappa_\nu(\xi-s)\dd\xi.
\end{equation}
The solution is understood in the natural-range sense. Whenever it is absolutely continuous, its leading term is the distributed-order Caputo operator above, so it is a classical zero-initial solution.

This includes genuinely continuous order distributions. For example, if $\dd\nu(\alpha)=\one_{[a,b]}(\alpha)\dd\alpha$, then
\begin{equation*}
k_\nu(t)=\int_a^b\frac{t^{-\alpha}}{\Gamma(1-\alpha)}\dd\alpha,
\qquad
\widehat\kappa_\nu(z)=
\begin{cases}
\dfrac{\log z}{z^b-z^a},&z\ne1,\\[6pt]
\dfrac1{b-a},&z=1.
\end{cases}
\end{equation*}
\Cref{prop:distributed} supplies the quantitative admissibility needed by \Cref{thm:sonine-volterra}; consequently, the resolvent representation of the unique natural-range solution applies to this continuous-order example without an additional kernel hypothesis and retains the prescribed bivariate perturbation.

\subsection{Regularized Prabhakar leading order}

For the remainder of the paper we use
\[
\AC^m[0,T]
:=\{u\in C^{m-1}[0,T]:u^{(m-1)}\in\AC[0,T]\}.
\]
Thus $u^{(m)}$ exists almost everywhere and belongs to $L^1(0,T)$; this is the convention used in the classical-range criterion below.

Let $R=\E_{\rho,\mu,\omega}^{\gamma}$ with kernel \eqref{eq:prabhakar-kernel}. By its power bound, asymptotics and \Cref{lem:injectivity}, $R$ is an admissible causal inverse. Put $A=R^{-1}$ on its natural range. For $0\le\beta<\mu$, causal convolution of the distribution kernel $h_{-\beta}$ with the Prabhakar kernel gives
\begin{equation}
D^\beta\E_{\rho,\mu,\omega}^{\gamma}
=\E_{\rho,\mu-\beta,\omega}^{\gamma}.
\label{eq:prabhakar-shift}
\end{equation}
Thus $D^\beta R=V_{q_\beta}$, where
\begin{equation*}
q_\beta(t)=t^{\mu-\beta-1}E_{\rho,\mu-\beta}^{\gamma}(\omega t^\rho)
=O(t^{\mu-\beta-1}).
\end{equation*}
For $v\in C[0,T]$, formula \eqref{eq:prabhakar-shift} is used at the kernel level to define the natural-range lower operator
\begin{equation}
B_\beta(Rv):=V_{q_\beta}v.
\label{eq:Bbeta}
\end{equation}
Injectivity of $R$ makes this definition unambiguous. Whenever the classical zero-initial Riemann--Liouville derivative of $Rv$ exists, it coincides with $B_\beta(Rv)$. Thus \eqref{eq:prabhakar-shift} does not assume classical differentiability of an arbitrary element of $R(C[0,T])$.

\begin{corollary}[Prabhakar--Volterra perturbation]\label{cor:prabhakar}
Let $\rho,\mu>0$, $\gamma,\omega\in\C$, $0\le\beta_j<\mu$, $p_j\in C[0,T]$, and $K\in C(\Delta_T)$. With $A=(\E_{\rho,\mu,\omega}^{\gamma})^{-1}$ on its natural range, the equation
\begin{equation}
Au+\sum_{j=1}^{N}p_j(t)B_{\beta_j}u
+\int_0^tK(t,s)u(s)\dd s=f(t)
\end{equation}
has the unique natural-range solution
\begin{equation}
u=\E_{\rho,\mu,\omega}^{\gamma}(I-V_{R_P})f,
\end{equation}
where $R_P$ is generated by
\begin{align}
G_P(t,s)
={}&\sum_{j=1}^{N}p_j(t)(t-s)^{\mu-\beta_j-1}
E_{\rho,\mu-\beta_j}^{\gamma}\bigl(\omega(t-s)^\rho\bigr)\\
&+\int_s^tK(t,\xi)(\xi-s)^{\mu-1}
E_{\rho,\mu}^{\gamma}\bigl(\omega(\xi-s)^\rho\bigr)\dd\xi.
\end{align}
Here $B_{\beta_j}$ denotes \eqref{eq:Bbeta}; on the classical zero-initial domain it is the usual $D^{\beta_j}$.
\end{corollary}

Let $m-1<\mu<m$. The strict upper inequality is used because the kernel-level formula below requires the positive parameter $m-\mu$. The endpoint $\mu=m$ may be handled separately by the usual integer-order limiting or composition convention; it is not needed for the range characterization formulated here. The regularized Prabhakar derivative is classically defined on $\AC^m[0,T]$ by
\begin{equation}
{}^C\!\D_{\rho,\mu,\omega}^{\gamma}u
=\E_{\rho,m-\mu,\omega}^{-\gamma}u^{(m)}.
\label{eq:regularized-prabhakar}
\end{equation}
It is the inverse of the Prabhakar integral after removal of the initial polynomial; see \cite{Giusti2020}. The boundary-jet extension makes that removal part of the operator domain itself. Indeed, the Prabhakar semigroup identity gives
\begin{equation}
R=\E_{\rho,\mu,\omega}^{\gamma}
=J^{m-1}\E_{\rho,\mu-m+1,\omega}^{\gamma}.
\label{eq:prabhakar-jet-factor}
\end{equation}
Thus $R(X)\subset\mathcal N_m$: every $Rv$ has continuous derivatives through order $m-1$, all vanishing at the origin. In particular,
\begin{equation}
\mathcal D_{\mathrm{Pr},m}
=\mathcal P_{m-1}\oplus R(X)
\label{eq:prabhakar-affine-domain}
\end{equation}
is a direct sum, and the coefficients of the polynomial component are exactly the ordinary initial jets.

For $P\in\mathcal P_{m-1}$ and $v\in X$, define
\begin{align}
{}^C\!\mathfrak D_{\rho,\mu,\omega}^{\gamma,[m]}(P+Rv)
&=v,
\label{eq:extended-prabhakar-leading}\\
{}^C\!D_{\beta,\partial}^{[m]}(P+Rv)
&={}^CD^\beta P+\E_{\rho,\mu-\beta,\omega}^{\gamma}v,
\qquad 0\le\beta<\mu.
\label{eq:extended-prabhakar-lower}
\end{align}
The finite-dimensional action in \eqref{eq:extended-prabhakar-lower} is explicit: for $0\le k<m$,
\begin{equation}
{}^CD^\beta e_k=
\begin{cases}
0,&k<\lceil\beta\rceil,\\[3pt]
\dfrac{t^{k-\beta}}{\Gamma(k+1-\beta)},&k\ge\lceil\beta\rceil.
\end{cases}
\label{eq:caputo-on-jets}
\end{equation}
With the convention $\lceil0\rceil=0$, this includes $\beta=0$. The operator in \eqref{eq:extended-prabhakar-leading} annihilates the boundary polynomial and agrees with \eqref{eq:regularized-prabhakar} wherever the latter is classically defined; \eqref{eq:extended-prabhakar-lower} similarly extends the lower Caputo operators from their classical domain.

\subsection{Boundary-jet Prabhakar equations with bivariate kernels}

The affine factorization now incorporates arbitrary prescribed initial jets directly. Given $\mathbf d=(d_0,\ldots,d_{m-1})$, write
\begin{equation*}
P_{\mathbf d}=\sum_{k=0}^{m-1}d_ke_k
\end{equation*}
and define the boundary forcing
\begin{equation}
F_{\mathbf d}(t)=f(t)
-\sum_{j=1}^{N}p_j(t){}^CD^{\beta_j}P_{\mathbf d}(t)
-\int_0^tK(t,s)P_{\mathbf d}(s)\dd s.
\label{eq:FP}
\end{equation}

\begin{theorem}[Prabhakar--Volterra equation with prescribed jets]\label{thm:prabhakar-jets}
Let $m-1<\mu<m$, $0\le\beta_j<\mu$, $p_j\in C[0,T]$, and $K\in C(\Delta_T)$. For every $f\in C[0,T]$ and every $\mathbf d\in\C^m$, the boundary-extended equation
\begin{equation}
{}^C\!\mathfrak D_{\rho,\mu,\omega}^{\gamma,[m]}u(t)
+\sum_{j=1}^{N}p_j(t){}^C\!D_{\beta_j,\partial}^{[m]}u(t)
+\int_0^tK(t,s)u(s)\dd s=f(t),
\qquad \varepsilon_ku=d_k,\quad 0\le k<m,
\label{eq:prabhakar-ivp}
\end{equation}
has a unique solution in the affine slice $P_{\mathbf d}+R(X)$. It is
\begin{equation}
u=P_{\mathbf d}
+\E_{\rho,\mu,\omega}^{\gamma}(I-V_{R_P})F_{\mathbf d},
\label{eq:prabhakar-solution}
\end{equation}
where $R_P$ is generated by the kernel $G_P$ in \Cref{cor:prabhakar}. The reduced kernel, and hence its resolvent, is independent of the prescribed jet vector $\mathbf d$.
\end{theorem}

\begin{proof}
Equation \eqref{eq:prabhakar-jet-factor} gives the transversality required in \eqref{eq:boundary-transversality}. Apply \Cref{thm:boundary-factorization} with $\mathcal B=\mathcal P_{m-1}$ and $\Lambda_jP={}^CD^{\beta_j}P$. Formulas \eqref{eq:prabhakar-shift} and \eqref{eq:extended-prabhakar-lower} give the same kernels $q_{\beta_j}$ as in the homogeneous problem, while \eqref{eq:boundary-forcing} becomes \eqref{eq:FP}. The boundary coordinates are the ordinary jets by $R(X)\subset\mathcal N_m$.
\end{proof}

The relation with classical regularity is exact. Set
\begin{equation*}
\AC_0^m[0,T]
=\{y\in\AC^m[0,T]:y^{(k)}(0)=0,\ 0\le k<m\}.
\end{equation*}

The standard Prabhakar semigroup and regularized left-inverse identities, as well as their operational-calculus use in solving Prabhakar equations, are developed in \cite{Giusti2020,RaniFernandez2022a,RaniFernandez2022b}. Those identities start on a prescribed classical or operational domain. The next proposition addresses the reverse domain question created by the present affine construction: given an arbitrary natural-range density $v\in C[0,T]$, it characterizes by an if-and-only-if condition exactly when $Rv$ has the classical zero-jet regularity $\AC_0^m$. Thus the claim is a range characterization, not a new semigroup identity or a restatement of the usual left-inverse formula.

\begin{proposition}[Exact classical-range criterion]\label{prop:classical-range}
For $v\in C[0,T]$,
\begin{equation}
Rv\in\AC_0^m[0,T]
\quad\Longleftrightarrow\quad
v=\E_{\rho,m-\mu,\omega}^{-\gamma}w
\ \text{a.e. on }(0,T)
\quad\text{for some }w\in L^1(0,T).
\label{eq:classical-range-equivalence}
\end{equation}
In that case $Rv=J^mw$ and
\begin{equation*}
{}^C\!\D_{\rho,\mu,\omega}^{\gamma}(Rv)=v
\quad\text{a.e. on }(0,T).
\end{equation*}
\end{proposition}

\begin{proof}
If $v=\E_{\rho,m-\mu,\omega}^{-\gamma}w$, the Prabhakar semigroup identity gives
\begin{equation*}
Rv
=\E_{\rho,\mu,\omega}^{\gamma}
 \E_{\rho,m-\mu,\omega}^{-\gamma}w
=J^mw\in\AC_0^m[0,T].
\end{equation*}
Conversely, if $Rv\in\AC_0^m[0,T]$, put $w=(Rv)^{(m)}\in L^1(0,T)$. The zero jets give $Rv=J^mw$, while the same semigroup identity gives
\begin{equation*}
J^mw=R\E_{\rho,m-\mu,\omega}^{-\gamma}w.
\end{equation*}
Injectivity of $R$ on $L^1(0,T)$ from \Cref{lem:injectivity} yields \eqref{eq:classical-range-equivalence}. Formula \eqref{eq:regularized-prabhakar} then gives the final identity.
\end{proof}

\begin{remark}[Relation to the classical Prabhakar inverse]
The standard left-inverse identity for the regularized Prabhakar derivative on its classical domain is established in the Prabhakar literature; see, for example, \cite{Giusti2020}. The new content of \Cref{prop:classical-range} is not that identity itself, but the exact equivalence that determines which elements of the larger affine natural-range construction actually belong to $\AC_0^m[0,T]$. Thus the proposition identifies the precise compatibility condition on the causal density $v$ that upgrades a natural-range solution to a classical one.
\end{remark}

\begin{corollary}[An explicit sufficient condition]\label{cor:AC-density}
Put $\delta=m-\mu\in(0,1)$. If $v\in\AC[0,T]$, then
$Rv\in\AC_0^m[0,T]$. More precisely, with the kernel notation
\eqref{eq:prabhakar-kernel}, define
\begin{equation}
w(t)=v(0)r_{\rho,1-\delta,\omega}^{\gamma}(t)
+\bigl(\E_{\rho,1-\delta,\omega}^{\gamma}v'\bigr)(t).
\label{eq:AC-density-w}
\end{equation}
Then $w\in L^1(0,T)$ and
\begin{equation*}
v=\E_{\rho,\delta,\omega}^{-\gamma}w
\qquad\text{a.e. on }(0,T).
\end{equation*}
Consequently, the condition in \Cref{prop:classical-range} is automatic for
absolutely continuous causal densities, although it is not automatic for the
continuous densities produced by the general resolvent theorem.
\end{corollary}

\begin{proof}
Because $1-\delta>0$, the first term in \eqref{eq:AC-density-w} is integrable;
the second is integrable by Young's inequality. The Prabhakar semigroup law
gives
\begin{equation*}
\E_{\rho,\delta,\omega}^{-\gamma}
r_{\rho,1-\delta,\omega}^{\gamma}=h_1=1
\quad\text{a.e.},
\end{equation*}
and
\begin{equation*}
\E_{\rho,\delta,\omega}^{-\gamma}
\E_{\rho,1-\delta,\omega}^{\gamma}v'
=J^1v'=v-v(0).
\end{equation*}
Adding these identities gives the asserted representation of $v$, and
\Cref{prop:classical-range} yields the conclusion.
\end{proof}

\begin{example}[The natural range is genuinely larger]\label{ex:strict-natural-range}
Take the Riemann--Liouville subcase $\gamma=0$, so that
$R=J^\mu$ and
$\E_{\rho,\delta,\omega}^{0}=J^\delta$.  We give an explicit continuous
density outside $J^\delta(L^1(0,T))$.  Put $b=2/\delta$ and, for $t>0$, define
\begin{equation}
w_*(t)=\frac{\sin(t^{-b})}{t},
\qquad
v_*(t)=\frac{1}{\Gamma(\delta)}
\lim_{\varepsilon\downarrow0}
\int_{\varepsilon}^{t}(t-s)^{\delta-1}w_*(s)\dd s,
\qquad v_*(0)=0.
\label{eq:explicit-nonrange-density}
\end{equation}
First, $w_*$ defines a causal distribution $W_*\in\D'_+(\R)$ by the
improper pairing
\begin{equation*}
\langle W_*,\varphi\rangle
:=\lim_{\varepsilon\downarrow0}
\int_\varepsilon^\infty \frac{\sin(t^{-b})}{t}\varphi(t)\dd t,
\qquad \varphi\in C_c^\infty(\R).
\end{equation*}
Indeed, if $\supp\varphi\subset(-\infty,M]$ for some $M>0$, the substitution $y=t^{-b}$
transforms the part near zero into
\begin{equation*}
\frac1b\int_{M^{-b}}^{\varepsilon^{-b}}
\frac{\sin y}{y}\varphi(y^{-1/b})\dd y,
\end{equation*}
which converges by integration by parts; the same estimate gives continuity
in the test-function topology.  Thus the integral in
\eqref{eq:explicit-nonrange-density} is the restriction of the causal
convolution $h_\delta*W_*$.

We next verify its continuity quantitatively.  With $\lambda=t^{-b}$, the
substitutions $s=tx$ and $y=x^{-b}$ give
\begin{equation}
v_*(t)=\frac{t^{\delta-1}}{b\Gamma(\delta)}
\int_1^\infty A(y)\sin(\lambda y)\dd y,
\qquad
A(y)=(1-y^{-1/b})^{\delta-1}y^{-1}.
\label{eq:oscillatory-density-scaled}
\end{equation}
On $1\le y\le2$, one has
$A(y)=(y-1)^{\delta-1}a(y)$ with $a\in C^1[1,2]$.  Splitting at
$y-1=\lambda^{-1}$, estimating the first part absolutely, and integrating
by parts on the second gives
\begin{equation*}
\left|\int_1^2A(y)e^{i\lambda y}\dd y\right|
\le C\lambda^{-\delta}.
\end{equation*}
On $[2,\infty)$, $A(y)=O(y^{-1})$ and $A'\in L^1(2,\infty)$, so one
integration by parts gives an $O(\lambda^{-1})$ bound.  Since
$\lambda^{-\delta}=t^{b\delta}$, $\lambda^{-1}=t^b$, $b\delta=2$, and
$b>2$, \eqref{eq:oscillatory-density-scaled} yields, for any fixed
$T_0\in(0,T]$,
\begin{equation*}
|v_*(t)|\le C t^{\delta-1+b\delta}=C t^{1+\delta}
\qquad (0<t\le T_0),
\end{equation*}
and hence $v_*(t)\to0$ as $t\downarrow0$.  The same integration-by-parts
estimates are uniform for $t$ in every compact subinterval of $(0,T]$;
therefore the truncated integrals converge locally uniformly and
$v_*\in C[0,T]$.
The particular choice $b=2/\delta$ is made for convenience, yielding the
clean estimate $v_*(t)=O(t^{1+\delta})$; the same argument gives continuity
at the origin whenever $b>(1-\delta)/\delta$.

On the other hand,
\begin{equation*}
\int_0^{T_0}|w_*(t)|\dd t
=\frac1b\int_{T_0^{-b}}^\infty\frac{|\sin y|}{y}\dd y
=\infty.
\end{equation*}
All convolutions in the following identity are taken in the algebra of causal
distributions $\D'_+(\R)$, where convolution is associative.
If $v_*=J^\delta g=h_\delta*g$ for some $g\in L^1(0,T)$, the causal
distribution identity $h_{-\delta}*h_\delta=h_0=\delta_0$ stated in
Section~2.1 gives
\begin{equation*}
g=h_{-\delta}*v_*=W_*
\qquad\text{in }\D'(0,T).
\end{equation*}
Since $W_*$ is represented by the smooth function $w_*$ away from zero,
this equality implies $g=w_*$ almost everywhere on every
$(\varepsilon,T)$, contradicting the preceding divergence.  Hence
\begin{equation*}
v_*\in C[0,T]\setminus J^\delta(L^1(0,T)).
\end{equation*}
Now put $u_*=J^\mu v_*$. Then $u_*$ belongs to the natural range
$J^\mu C[0,T]$, but \Cref{prop:classical-range} gives
$u_*\notin\AC_0^m[0,T]$. Hence
\begin{equation*}
J^\mu C[0,T]\cap\AC_0^m[0,T]
\subsetneq J^\mu C[0,T].
\end{equation*}
This strict inclusion isolates the domain information supplied by
\Cref{prop:classical-range}: the proposition does not merely reapply the
standard left-inverse identity on a domain already known to be classical.
\end{example}

For the solution in \eqref{eq:prabhakar-solution}, let
\begin{equation*}
v_{\mathbf d}=(I-V_{R_P})F_{\mathbf d}.
\end{equation*}
It is the classical solution of the regularized-Prabhakar initial-value problem with $u^{(k)}(0)=d_k$ precisely when $v_{\mathbf d}$ satisfies the equivalent conditions in \Cref{prop:classical-range}. Thus the affine theorem is unconditional on $P_{\mathbf d}+R(X)$, while \Cref{prop:classical-range} identifies its classical subdomain exactly.

In particular, taking $N=0$ and replacing $K$ by $-\lambda K$ gives
\begin{equation}
{}^C\!\mathfrak D_{\rho,\mu,\omega}^{\gamma,[m]}u(t)
=f(t)+\lambda\int_0^tK(t,s)u(s)\dd s,
\qquad \varepsilon_ku=d_k,\quad 0\le k<m.
\end{equation}
The corresponding sign-reversed reduced kernel is
\begin{equation*}
H_P(t,s)=\lambda\int_s^tK(t,\xi)(\xi-s)^{\mu-1}
E_{\rho,\mu}^{\gamma}\bigl(\omega(\xi-s)^\rho\bigr)\dd\xi.
\end{equation*}
If $S_P=H_P+H_P^{(*2)}+H_P^{(*3)}+\cdots$, then
\begin{equation}
u=P_{\mathbf d}+\E_{\rho,\mu,\omega}^{\gamma}(I+V_{S_P})F_{\mathbf d}.
\end{equation}
The independently prescribed kernel remains bivariate throughout the reduction. On the exact classical subdomain of \Cref{prop:classical-range}, the construction provides a resolvent representation for the corresponding general-kernel problem considered in \cite{Eshaghi2021}, without first imposing convolution structure.

\section{Discussion and conclusion}

The main algebraic result is the unique finite normal form for polynomial-coefficient fractional powers modulo the finite polynomially weighted diagonal-flat convolution ideal $\Vflat$. Polynomial coefficients provide the local finiteness needed for fractional Leibniz rewriting, diagonal flatness makes this specific Volterra remainder invariant under every order in $\Gamma$, and distributional separation prevents cancellation between generalized powers and the remainder. The restrictions are sharp for the stated finite architecture: finite reordering of a smooth multiplier with $J$ modulo even the full smooth diagonal-flat remainder forces that multiplier to be polynomial, while $\Vflat$ is the smallest two-sided ideal containing the flat convolution generators. The resulting order is additive; in particular, $\A_\Gamma/\Vflat$ has no zero divisors and the leading symbol is multiplicative. These sharpness results delimit the finite architecture from arbitrary smooth multipliers and arbitrary smooth Volterra kernels, for which finite closure fails. The flat space $\M$ serves as the maximal homogeneous $D$--$J$ inverse core, while the later natural-range and affine constructions carry the equation-level solutions.

The analytic part of the paper connects the finite-normal-form algebra to causal inversion without turning the full bivariate perturbation into an element of that algebra.  Proposition~\ref{prop:core-range-bridge} gives the link explicitly.  For a monic normal form, its unique leading power $D^\alpha$ determines the right factor $J^\alpha$; all lower powers become smoothing factors, and the resulting identity extends compatibly from $\M$ to $J^\alpha C[0,T]$. The monic scope is exact: a nowhere-vanishing general leading coefficient may be normalized at the equation level but can leave the polynomial algebra, while a vanishing coefficient obstructs the global reduction. The $L^1$ completion then enlarges the negative-order convolution smoothing sector, where it contains integrable Sonine kernels and the norm-convergent Prabhakar expansion. For an injective smoothing inverse, compatible lower operators and an independent $K(t,s)$ reduce to one weakly singular second-kind Volterra equation. The reduced kernel remains explicitly bivariate, and classical Volterra theory supplies its resolvent on every finite interval without a norm-smallness assumption.

Endpoint data are handled by a finite boundary block rather than by enlarging the homogeneous normal-form theorem. On $\mathcal P_{m-1}\oplus R(C[0,T])$, the initial polynomial and causal density are separated canonically, and prescribed jets change only the forcing while leaving the reduced resolvent kernel unchanged. Thus the successive enlargements have distinct roles: $\M$ carries the endpoint-free finite normal form, $R(C[0,T])$ is the homogeneous solution domain, and $\mathcal P_{m-1}\oplus R(C[0,T])$ carries prescribed finite jets. For power-law and regularized Prabhakar equations these boundary coordinates are the ordinary initial derivatives. In the Prabhakar case, \Cref{prop:classical-range} further identifies exactly when the unconditional affine natural-range solution lies in the classical $\AC^m$ domain. The explicit formula \eqref{eq:AC-density-w} makes the test automatic for absolutely continuous densities, whereas \Cref{ex:strict-natural-range} proves in the Riemann--Liouville subcase that the continuous natural range is strictly larger than its classical $\AC_0^m$ part.

The Sonine and distributed-order realizations show that the factorization is not tied to a single power law, while the nonconvolution examples demonstrate that an independently prescribed continuous bivariate memory term survives the reduction. In comparison with the settings treated in \cite{Eshaghi2021,FernandezRestrepo2022,Guzoglu2026}, the present reduction simultaneously retains continuous variable coefficients, compatible lower-order fractional operators, and an independently prescribed continuous bivariate Volterra kernel. The framework therefore separates three roles that are often mixed: a fractional operator algebra with finite normal forms, a completed convolution smoothing sector, and an external bivariate Volterra perturbation. Natural extensions include matrix-valued kernels, broader locally finite coefficient classes, boundary modules for more general Sonine traces, and a systematic study of units and localizations of $\A_\Gamma/\Vflat$.

\end{document}